\documentclass[10pt]{amsart}

\usepackage[latin1]{inputenc}
\usepackage[T1]{fontenc}

\usepackage[theoremfont]{baskervillef}
\usepackage[baskerville, varg]{newtxmath}
\usepackage[scr=boondox,cal=boondoxo,bb=boondox,frak=euler]{mathalfa}

\usepackage{hyperref}

\usepackage{amsfonts}
\usepackage{cleveref}

\usepackage[style=alphabetic,natbib=true,backend=biber,maxbibnames=10]{biblatex}

\usepackage{graphicx}
\usepackage{multirow}
\usepackage{enumitem}
\usepackage{longtable}

\allowdisplaybreaks[4]

\makeatletter

\newsavebox\myboxA
\newsavebox\myboxB
\newlength\mylenA

\newcommand*\xoverline[2][0.75]{%
    \sbox{\myboxA}{$\m@th#2$}%
    \setbox\myboxB\null
    \ht\myboxB=\ht\myboxA%
    \dp\myboxB=\dp\myboxA%
    \wd\myboxB=#1\wd\myboxA
    \sbox\myboxB{$\m@th\overline{\copy\myboxB}$}
    \setlength\mylenA{\the\wd\myboxA}
    \addtolength\mylenA{-\the\wd\myboxB}%
    \ifdim\wd\myboxB<\wd\myboxA%
       \rlap{\hskip 0.5\mylenA\usebox\myboxB}{\usebox\myboxA}%
    \else
        \hskip -0.5\mylenA\rlap{\usebox\myboxA}{\hskip 0.5\mylenA\usebox\myboxB}%
    \fi}

\makeatother

\newcommand{\A}{\mathcal{A}}
\newcommand{\B}{\mathcal{B}}
\newcommand{\Z}{\mathbb{Z}}
\newcommand{\odd}{{\mathrm{odd}}}
\newcommand{\even}{{\mathrm{even}}}
\newcommand{\hyp}{{\mathrm{hyp}}}
\renewcommand{\H}{\mathcal{H}}
\newcommand{\Q}{\mathcal{Q}}

\newcommand{\RV}{\mathrm{RV}}

\newcommand{\Sp}{\mathrm{Sp}}

\newcommand{\Id}{\mathrm{Id}}
\newcommand{\tr}{{\mbox{\raisebox{0.33ex}{\scalebox{0.6}{$\intercal$}}}}}
\let\Sigmaaux\Sigma
\renewcommand{\Sigma}{\mbox{\scalebox{1.05}{$\Sigmaaux$}}}
\newcommand{\RR}{\mathcal{R}}
\newcommand{\C}{\mathbb{C}}

\renewcommand{\Re}{\mathrm{Re}}
\renewcommand{\Im}{\mathrm{Im}}
\renewcommand{\S}{\mathcal{S}}
\newcommand{\reg}{\mathrm{reg}}
\newcommand{\irr}{\mathrm{irr}}
\newcommand{\nonhyp}{\mathrm{nonhyp}}

\newtheorem{thm}{Theorem}[section]
\newtheorem{conv}{Convention}
\newtheorem{lem}[thm]{Lemma}

\newtheorem{cor}[thm]{Corollary}

\theoremstyle{definition}
\newtheorem{defn}[thm]{Definition}
\theoremstyle{remark}
\newtheorem{rem}[thm]{Remark}

\DeclareFieldFormat{title}{\mkbibquote{#1}}

\begin{document}

	\title{Simplicity of the Lyapunov spectra of certain quadratic differentials}
	
	\author{Rodolfo Gutiérrez-Romo}
	\address{Institut de Mathématiques de Jussieu -- Paris Rive Gauche, UMR 7586, Bâtiment Sophie Germain, 75205 PARIS Cedex 13, France.} \email{rodolfo.gutierrez@imj-prg.fr}
	
	\begin{abstract}
		We prove that the ``plus'' Rauzy--Veech groups of all connected components of the strata of meromorphic quadratic differentials defined on Riemann surfaces of genus at least one having at most simple poles and at least three singularities (zeros or poles), not all of even order, are finite-index subgroups of their ambient symplectic groups. This shows that the ``plus'' Lyapunov spectrum of such strata is simple. Moreover, we show that the index of the ``minus'' Rauzy--Veech group is also finite for connected components of strata satisfying the same conditions and having exactly two singularities of odd order. This shows that the ``minus'' Lyapunov spectrum of such strata is simple.
	\end{abstract}
	
	\maketitle
	\markboth{R. GUTIÉRREZ-ROMO}{SIMPLICITY OF THE LYAPUNOV SPECTRA OF CERTAIN QUADRATIC DIFFERENTIALS}
	\section{Introduction}
	Several important dynamical and geometrical properties of typical interval exchange transformations and translation flows, such as the deviation of ergodic averages \cite{Z:deviations, F:deviations}, can be characterized in terms of the Lyapunov spectrum of their embodying stratum with respect to the Masur--Veech measures. The structure of such spectrum was made precise by Kontsevich and Zorich, who conjectured that it is always simple, and by Avila and Viana, who famously established this conjecture in 2007 \cite{AV:KZ_conjecture}.
	
	For the case of half-translation flows (or, equivalently, quadratic differentials), the situation is much less clear. There are two families of exponents, usually called ``plus'' and ``minus'' \cite{EKZ:sum}. Some dynamical properties, such as the deviation of ergodic averages \cite{F:deviations,EKZ:sum} are described by the ``minus'' Lyapunov exponents, which are defined in terms of the canonical orientable double cover of the quadratic differential. Other dynamical properties, such as the diffusion rate of windtree models \cite{DHL:windtree, DZ:windtree}, are controlled by the ``plus'' Lyapunov exponents. It is known that the smallest non-negative ``plus'' and ``minus'' exponents are positive and that the two largest ``minus'' exponents are distinct \cite{T:hyperbolicity}. However, the simplicity is not known for neither family in higher genus.
	
	In this article we study the Rauzy--Veech groups of some strata of meromorphic quadratic differentials with at most simple poles. We restrict to the case of connected components of strata having at least three singularities and at least one of odd order. Observe that, from the classification of the connected components of the strata of the moduli space quadratic differentials \cite{L:connected_components,CM:low_genus}, all such strata are connected except for some strata of the form $\Q(4j+2,2k-1,2k-1)$ or $\Q(2j-1,2j-1,2k-1,2k-1)$ for integers $j,k\geq 0$ and the exceptional strata $\Q(6,3,-1)$ and $\Q(3,3,3,-1)$ in genus $3$, and $\Q(3,3,3,3)$ and $\Q(6,3,3)$ in genus $4$.
	
	Our main result is the following:
	
	\smallbreak
	\begin{thm} \label{thm:main}
		The ``plus'' Rauzy--Veech group of any connected component of a stratum of meromorphic quadratic differentials defined on genus-$g$ Riemann surfaces with $g \geq 1$ having at most simple poles and at least three singularities (zeros or poles), not all of even order, and the ``minus'' Rauzy--Veech group of connected components of strata satisfying the same conditions and having exactly two singularities of odd order are finite-index subgroups of their ambient symplectic groups. More precisely:
		\begin{itemize}
			\item The ``plus'' and ``minus'' Rauzy--Veech groups of $\Q(4j+2,2k-1,2k-1)^\hyp$ contain the Rauzy--Veech group of $\H(2g - 2)^\hyp$ for every $j, k \geq 0$.
			\item The ``plus'' Rauzy--Veech group of $\Q(2j-1,2j-1,2k-1,2k-1)^\hyp$ contains the Rauzy--Veech group of $\H(g - 1, g - 1)^\hyp$ for every $j \geq 1$ and $k \geq 0$.
			\item The ``plus'' and ``minus'' Rauzy--Veech groups of $\Q(2,3,3)^\nonhyp$ contain the Rauzy--Veech group of $\H(4)^\odd$, so their indices are at most $28$.
			\item The ``plus'' Rauzy--Veech group of $\Q(3,3,-1,-1)^\nonhyp$ is equal to its entire ambient symplectic group.
			\item In any other case, provided the conditions on the singularities are satisfied, the Rauzy--Veech groups are equal to their entire ambient symplectic groups, except if $g = 2$ where they contain the Rauzy--Veech group of $\H(2)$ and their indices are thus at most $6$.
		\end{itemize}
	\end{thm}
	Observe that most Rauzy--Veech groups of connected components of strata satisfying our hypotheses are equal to their entire ambient symplectic groups. The possible exceptions are hyperelliptic components and some other components in genus two and three.
	
	\smallbreak
	
	To prove the Kontsevich--Zorich conjecture, Avila and Viana established a general criterion for the simplicity of the Lyapunov spectrum of symplectic cocycles \cite{AV:simplicity,AV:KZ_conjecture}. In the case of the Teichmüller flow, this general criterion amounts to showing that the underlying monoid is \emph{pinching} and \emph{twisting}, which is almost automatic in the case that the group arising from the monoid is a finite-index subgroup of the symplectic group. Therefore, we obtain the following corollary:
	
	\smallbreak
	\begin{cor} \label{cor:simplicity}
		The ``plus'' Lyapunov spectrum of any connected component of any stratum of meromorphic quadratic differentials defined on Riemann surfaces of genus at least one having at most simple poles and at least three singularities (zeros or poles), not all of even order, is simple. Moreover, the ``minus'' Lyapunov spectrum of connected components of strata satisfying the same conditions and having exactly two singularities of odd order is also simple.
	\end{cor}
	\smallbreak
	
	In order to prove \Cref{thm:main}, we rely on the classification of Rauzy--Veech groups of Abelian differentials \cite{AMY:hyperelliptic,G:zariski}. Indeed, we generalize the notion of simple extension defined by Avila and Viana \cite[Section 5.2]{AV:KZ_conjecture} to find explicit combinatorial adjacencies between the strata mentioned in \Cref{thm:main} and some strata of Abelian differentials. Then, we use the fact that the indices of the Rauzy--Veech groups of Abelian differentials are finite.
	
	The article is organized as follows. In \Cref{sec:preliminaries} we review the basic definitions and background. In \Cref{sec:extensions} we generalize simple extensions to the case of generalized permutations. In \Cref{sec:rauzy-veech} we define the ``plus'' and ``minus'' Rauzy--Veech groups. In \Cref{sec:proof} we prove \Cref{thm:main}. Finally, in \Cref{sec:simplicity} we show that this theorem implies the simplicity of the Lyapunov spectra in the relevant cases.
	
	\section{Rauzy--Veech induction} \label{sec:preliminaries}
	
	\subsection{Generalized permutations} In this section we will recall the Rauzy--Veech induction algorithm on generalized permutations. We refer the reader to the work by Boissy and Lanneau \cite{BL:rauzy_veech} and by Avila and Resende \cite{AR:exponential} for more details on the algorithm on quadratic differentials and to the lecture notes by Yoccoz \cite{Y:pisa} and the survey by Viana \cite{V:iet} for more details about the Abelian case.
	
	Let $\A$ be a finite set of cardinality $d \geq 2$ and $\ell, m$ be positive integers satisfying $\ell + m = d$. A \emph{generalized permutation} of type $(\ell, m)$ is a two-to-one map $\pi \colon \{1, \dotsc, 2d\} \to \A$. We usually write such a map by a table
	{\small\[
		\pi =
		\begin{pmatrix}
			\pi(1) & \pi(2) & \cdots & \pi(\ell) \\
			\pi(\ell+1) & \pi(\ell+2) & \cdots & \pi(\ell+m)
		\end{pmatrix}.
	\]}
	An involution $\sigma \colon \{1, \dotsc, 2d\} \to \{1, \dotsc, 2d\}$ is defined naturally from a generalized permutation by the rules $\sigma(i) \neq i$ and $\pi(\sigma(i)) = \pi(i)$ for every $i \in \{1, \dotsc, 2d\}$. That is, $\{i, \sigma(i)\}$ are the two positions of the letter $\pi(i) = \pi(\sigma(i))$. We say that a letter is \emph{duplicate} if both of its occurrences are in the same row.
	
	We will treat ``genuine'' permutations as special cases of generalized permutations. That is, a \emph{permutation} is a generalized permutation such that $\ell = m$ and $\sigma(i) > \ell$ for every $1 \leq i \leq \ell$, that is, having no duplicate letters. We say that a generalized permutation is a strict generalized permutation if it is not a permutation. We will also assume the following \cite[Convention 2.7]{BL:rauzy_veech}:
	\smallbreak
	\begin{conv}\label{conv:2.7}
		Every strict generalized permutation contains duplicate letters in both rows.
	\end{conv}
	\smallbreak
	The importance of this convention lies in that it is necessary for the existence of a suspension of a strict generalized permutation.
	
	A \emph{decomposition} of a generalized permutation $\pi$ is a way of writing it as
	\[
		\pi = \left(\begin{array}{c|c|c}
			F_{\mathrm{tl}} & *** & F_{\mathrm{tr}} \\\hline
			F_{\mathrm{bl}} & *** & F_{\mathrm{br}}
		\end{array}\right)
	\]
	where $F_{\mathrm{tl}}, F_{\mathrm{tr}}, F_{\mathrm{bl}}, F_{\mathrm{br}}$ are (possibly empty) subsets of $\A$. This notation means that there exist $1 \leq i_1 \leq i_2 \leq \ell < i_3 \leq i_4 \leq \ell + m$ such that
	\begin{itemize}
		\item $F_{\mathrm{tl}} = \{\pi(1), \dotsc, \pi(i_1)\}$;
		\item $F_{\mathrm{tr}} = \{\pi(i_2), \dotsc, \pi(\ell)\}$;
		\item $F_{\mathrm{bl}} = \{\pi(\ell+1), \dotsc, \pi(i_3)\}$;
		\item $F_{\mathrm{br}} = \{\pi(i_4), \dotsc, \pi(\ell+m)\}$.
	\end{itemize}
	Once a decomposition is clear from context, we refer to $F_{\mathrm{tl}}, F_{\mathrm{tr}}, F_{\mathrm{bl}}, F_{\mathrm{br}}$ as the top-left, top-right, bottom-left and bottom-right corners of $\pi$, respectively.
	
	Let $\pi$ be a strict generalized permutation. We say that $\pi$ is \emph{reducible} if there exists a decomposition
	\[
		\pi = \left(\begin{array}{c|c|c}
			A \cup B & *** & D \cup B \\\hline
			A \cup C & *** & D \cup C 
		\end{array}\right)
	\]
	where $A,B,C,D$ are disjoint (possibly empty) subsets of $\A$ satisfying one of the following conditions:
	\begin{itemize}
		\item no corner is empty;
		\item there is exactly one empty corner and it is on the left;
		\item there are exactly two empty corners and they are on the same side.
	\end{itemize}
	Otherwise, we say that it is \emph{irreducible}. On the other hand, if $\pi$ is a permutation we will use the usual definition of irreducibility: $\pi(\{1, \dots, k\}) \neq \pi(\{\ell+1, \dots, \ell+k\})$ for each $k < d$.
	
	One has that a generalized permutation stems from the directional flow of a quadratic differential on a Riemann surface or, equivalently, admits a suspension datum (defined below) if and only if it is irreducible and satisfies \Cref{conv:2.7} \cite[Theorem 3.2]{BL:rauzy_veech}. From now on, we will always assume that a generalized permutation satisfies these properties unless explicitly stated otherwise.
	
	A \emph{suspension datum} for a generalized permutation $\pi$ is a collection $(\zeta_\alpha)_{\alpha \in \A}$ of complex numbers satisfying:
	\begin{itemize}
		\item $\Re(\zeta_\alpha) > 0$ for each $\alpha \in \A$;
		\item $\sum_{1 \leq j \leq i} \Im(\zeta_{\pi(j)}) > 0$ for each $1 \leq i \leq \ell - 1$;
		\item $\sum_{1 \leq j \leq i} \Im(\zeta_{\pi(\ell+j)}) < 0$ for each $1 \leq i \leq m - 1$;
		\item $\sum_{1 \leq i \leq \ell} \zeta_{\pi(i)} = \sum_{1 \leq i \leq m} \zeta_{\pi(\ell+i)}$.
	\end{itemize}
	
	Suspension data may not necessarily define a suitable polygon. That is, the broken lines defined by the suspension datum may intersect at points different from $0$ and $\sum_{1 \leq i \leq \ell} \zeta_{\pi(i)} = \sum_{1 \leq i \leq m} \zeta_{\pi(\ell+i)}$. Nevertheless, it is always possible to construct another suspension datum from $\zeta$ which defines a suitable polygon \cite[Lemma 2.12]{BL:rauzy_veech}. For a generalized permutation $\pi$, we choose any suspension datum $\zeta$ that admits a suitable polygon and define $P_\pi$ to be such polygon and $M_\pi$ to be the half-translation surface obtained by identifying the equally-labelled sides of $P_\pi$ by translations and/or central symmetries. We put $e_\pi = \sum_{1 \leq i \leq \ell} \zeta_{\pi(i)} = \sum_{1 \leq i \leq m} \zeta_{\pi(\ell+i)}$. The arbitrary choice of $P_\pi$ is not a problem, since the notions that we will define and use are homological. Moreover, since our strategy consists of exploiting genus-preserving adjacencies with Abelian strata, we will assume that the genus of $M_\pi$ is at least one.
	
	\subsection{Rauzy--Veech induction} The Rauzy--Veech induction algorithm for generalized permutations, defined by Boissy and Lanneau \cite[Section 2.2]{BL:rauzy_veech}, consists of at most two operations $R_{\mathrm{t}}$ and $R_{\mathrm{b}}$ on $\pi$, which we call \emph{top} and \emph{bottom}.
	
	If $\sigma(\ell) > \ell$, then $R_{\mathrm{t}}$ is the type-$(\ell, m)$ generalized permutation defined as:
	\[
		R_{\mathrm{t}}(\pi)(i) =
		\begin{cases}
			\pi(i) & i \leq \sigma(\ell) \\
			\pi(\ell+m) & i = \sigma(\ell) + 1 \\
			\pi(i-1) & \text{otherwise};
		\end{cases}
	\]
	if $\sigma(\ell) < \ell$ and there exists duplicate letter in the bottom row of $\pi$ which is not the last letter, then $R_{\mathrm{t}}$ is the type-$(\ell+1, m-1)$ generalized permutation defined as:
	\[
		R_{\mathrm{t}}(\pi)(i) =
		\begin{cases}
			\pi(i) & i < \sigma(\ell) \\
			\pi(\ell+m) & i = \sigma(\ell) \\
			\pi(i-1) & \text{otherwise};
		\end{cases}
	\]
	and, in any other case, $R_{\mathrm{t}}$ is not defined on $\pi$. When a top operation is defined, we call $\pi(\ell)$ the \emph{winner} and $\pi(\ell+m)$ the \emph{loser} of the operation.
	
	Similarly, if $\sigma(\ell+m) < \ell$, then $R_{\mathrm{b}}$ is the type-$(\ell, m)$ generalized permutation defined as:
	\[
		R_{\mathrm{b}}(\pi)(i) =
		\begin{cases}
			\pi(\ell) & i = \sigma(\ell+m) + 1 \\
			\pi(i-1) & \sigma(\ell+m) + 1 < i \leq \ell \\
			\pi(i) & \text{otherwise};
		\end{cases}
	\]
	if $\sigma(\ell+m) > \ell$ and there exists duplicate letter in the top row of $\pi$ which is not the last letter, then $R_{\mathrm{b}}$ is the type-$(\ell-1, m+1)$ generalized permutation defined as:
	\[
		R_{\mathrm{b}}(\pi)(i) =
		\begin{cases}
			\pi(i+1) & \ell \leq i < \sigma(\ell+m) + 1 \\
			\pi(\ell) & i = \sigma(\ell+m) - 1 \\
			\pi(i) & \text{otherwise};
		\end{cases}
	\]
	and, in any other case, $R_{\mathrm{b}}$ is not defined on $\pi$. When a bottom operation is defined, we call $\pi(\ell+m)$ the \emph{winner} and $\pi(\ell)$ the \emph{loser} of the operation.
	
	Observe that, if $\pi$ is irreducible, then at least one of these operations is defined on $\pi$. Moreover, $R_{\mathrm{t}}(\pi)$, $R_{\mathrm{b}}(\pi)$ are also irreducible if they are defined.
	
	Consider now the directed graph whose vertices are the irreducible generalized permutations and such that $\pi \to \pi'$ is an edge if $R_{\mathrm{t}}(\pi) = \pi'$ or $R_{\mathrm{b}}(\pi) = \pi'$. The connected (or, equivalently, strongly connected) components of this graph are called \emph{Rauzy classes}. They are in a finite-to-one correspondence with the connected components of the strata of the moduli space of quadratic differentials \cite[Theorem D]{BL:rauzy_veech}.
	
	We will also use some results of Avila and Resende's work \cite{AR:exponential}, which uses a slightly different formalism for generalized permutations. Indeed, let $\A$ be an alphabet with $2d$ equipped with a fixed-point-free involution $\iota \colon \A \to \A$. Let $\ast \notin \A$ be a letter. We say that a bijection $\tau \colon \A \cup \{\ast\} \to \{1, \dotsc, 2d+1\}$ is a \emph{permutation with involution} if $\iota(\A_{\mathrm{l}}) \not\subseteq \A_{\mathrm{r}}$ and $\iota(\A_{\mathrm{r}}) \not\subseteq \A_{\mathrm{l}}$, where $\A_{\mathrm{l}} = \{\alpha \in \A \ \mid\ \tau(\alpha) < \tau(\ast)\}$ and $\A_{\mathrm{r}} = \{\alpha \in \A \ \mid\ \tau(\alpha) > \tau(\ast)\}$. We write such a bijection as a table
	\[
		\tau = \begin{pmatrix}
			\tau^{-1}(1) & \tau^{-1}(2) & \dotsc & \tau^{-1}(2d + 1)
		\end{pmatrix}.
	\]
	For any generalized permutation $\pi$ on an alphabet $\A$, we can define a permutation with involution $\tau$ on the alphabet $\A \times \{0, 1\}$ in the following way: consider the involution $\iota \colon \A \times \{0, 1\} \to \A \times \{0, 1\}$ defined as $\iota(\alpha, \varepsilon) = (\alpha, 1 - \varepsilon)$ for each $\alpha \in \A$ and $\varepsilon \in \{0, 1\}$ and let
	\[
		\tau = \begin{pmatrix}
			(\pi(\ell + m), \varepsilon_{\ell + m}) & \cdots & (\pi(\ell + 1), \varepsilon_{\ell + 1}) & \ast & (\pi(1), \varepsilon_1) & \dotsc & (\pi(\ell), \varepsilon_\ell)
		\end{pmatrix}.
	\]
	where the $\varepsilon_j \in \{0, 1\}$ are chosen so $\varepsilon_j = 1 - \varepsilon_{\sigma(j)}$ for every $j \in \{1, \dotsc, 2d\}$. Then, \Cref{conv:2.7} is equivalent to $\iota(\A_{\mathrm{l}}) \not\subseteq \A_{\mathrm{r}}$ and $\iota(\A_{\mathrm{r}}) \not\subseteq \A_{\mathrm{l}}$, so we obtain a bijection between generalized permutations satisfying \Cref{conv:2.7} and permutations with involution (up to exchanging letters in the same orbit of $\iota$). Moreover, a straightforward computation shows that the right and left operations on permutations with involution correspond, respectively, to top and bottom operations on generalized permutations.
	
	\section{Simple extensions} \label{sec:extensions}
	
	\begin{defn} \label{def:extension}
		Let $\tau$ be an irreducible generalized permutation on an alphabet $\B$ not containing $\alpha$. We say that a generalized permutation $\pi$ on the alphabet $\B \cup \{\alpha\}$ is a \emph{simple extension} of $\tau$ if $\tau$ is obtained from $\pi$ by erasing the letter $\alpha$ and the following conditions hold:
		\begin{itemize}
			\item $\alpha$ is not at the end of any row of $\pi$;
			\item at least one occurrence of $\alpha$ is not at the beginning of a row of $\pi$.
		\end{itemize}
	\end{defn}
	Observe that any simple extension of a strict generalized permutation also satisfies \Cref{conv:2.7}. On the other hand, a simple extension of a permutation satisfies this convention only when it is a permutation as well.
	
	We have that irreducibility is preserved by simple extensions of strict generalized permutations:
	
	\begin{lem}
		Let $\pi$ be a simple extension of a strict generalized permutation $\tau$ obtained by inserting a letter $\alpha$. If $\tau$ is irreducible, then $\pi$ is irreducible as well.
	\end{lem}
	
	\begin{proof}
		We will prove that if $\pi$ is reducible, then the letter $\alpha$ was inserted in some positions that are forbidden by definition of simple extension. Assume then that $\pi$ is reducible, so let
		\[
			\pi = \left(\begin{array}{c|c|c}
				A \cup B & *** & D \cup B \\\hline
				A \cup C & *** & D \cup C 
			\end{array}\right),
			\text{ with } A,B,C,D \text{ disjoint subsets of } \mathcal{A},
		\]
		be a decomposition as in the definition of reducibility.
		
		Let $A' = A \setminus \{\alpha\}$, $B' = B \setminus \{\alpha\}$, $C' = C \setminus \{\alpha\}$ and $D' = D \setminus \{\alpha\}$. We have that
		\[
			\tau = \left(\begin{array}{c|c|c}
				A' \cup B' & ***' & D' \cup B' \\\hline
				A' \cup C' & ***' & D' \cup C'
			\end{array}\right)
		\]
		and that this decomposition does not satisfy the definition of reducibility. That is: at least one corner is empty; if there is exactly one empty corner, it is on the right; and, if there are exactly two empty corners, they are on different sides. Observe that, in particular, $\alpha \in A \cup B \cup C \cup D$ and that $\tau$ has more empty corners than $\pi$. Therefore, the set in $\{A, B, C, D\}$ containing $\alpha$ is actually equal to $\{\alpha\}$.
		
		We consider two cases:
		\begin{enumerate}
			\parindent=0pt
			\item If $A = \{\alpha\}$, then the left corners of $\pi$ are not empty. If both right corners of $\pi$ were non-empty, then both right corners of $\tau$ would be non-empty as well, which is not possible. By definition of reducibility, it is not possible that exactly one right corner of $\pi$ is empty. Therefore, both of its right corners are empty. We conclude that $\alpha$ was inserted at the beginning of both rows.
			\item If $B = \{\alpha\}, C = \{\alpha\}$ or $D = \{\alpha\}$, then one of the right corners must be equal to $\{\alpha\}$ since, otherwise, $\tau$ would not have more empty corners than $\pi$. Therefore, $\alpha$ was inserted at the end of a row.
		\end{enumerate}
	\end{proof}
	
	If $\pi$ is a simple extension of $\tau$ and $\eta$ is an arrow in the Rauzy class of $\tau$, we define the path $\gamma = \mathcal{E}_*(\eta)$ in the Rauzy class of $\pi$ starting at $\pi$ as follows:
	\begin{enumerate}
		\item If $\eta$ is of top type and $\alpha$ is the next-to-last letter in the bottom row of $\tau$, then $\gamma$ consists of two top arrows if the occurrences of $\alpha$ are not consecutive, and of three top arrows if they are. \label{case1}
		\item If $\eta$ is of bottom type and $\alpha$ is the next-to-last letter in the top row of $\tau$, then $\gamma$ consists of two bottom arrows if the occurrences of $\alpha$ are not consecutive, and of three bottom arrows if they are. \label{case2}
		\item Otherwise, $\gamma$ consists of a single arrow of the same type of $\eta$. \label{case3}
	\end{enumerate}
	
	\begin{lem}
		Using the previous notation, $\gamma$ is well-defined and its end $\pi'$ is a simple extension of the end $\tau'$ of $\eta$.
	\end{lem}
	
	\begin{proof}
		We consider the three cases separately. In case \eqref{case1}, we may have that the winning letter $\ast$ of $\eta$ occurs in both rows. If the occurrences of $\alpha$ are not consecutive, then
		{\small\[
			\eta =
			\begin{pmatrix}
				\cdot & \cdot & \cdot & \cdot & \ast \\
				\cdot & \ast & \cdot & \cdot & \beta
			\end{pmatrix} \to
			\begin{pmatrix}
				\cdot & \cdot & \cdot & \cdot & \ast \\
				\cdot & \ast & \beta & \cdot & \cdot 
			\end{pmatrix}
		\]}
		{\small\[
			\gamma =
			\begin{pmatrix}
				\cdot & \cdot & \cdot & \cdot & \cdot & \ast \\
				\cdot & \ast & \cdot & \cdot & \alpha & \beta
			\end{pmatrix} \to
			\begin{pmatrix}
				\cdot & \cdot & \cdot & \cdot & \cdot & \ast \\
				\cdot & \ast & \beta & \cdot & \cdot & \alpha
			\end{pmatrix} \to
			\begin{pmatrix}
				\cdot & \cdot & \cdot & \cdot & \cdot & \ast \\
				\cdot & \ast & \alpha & \beta & \cdot & \cdot 
			\end{pmatrix}
		\]}
		Contrarily, if they are consecutive, then
		{\small\[
			\eta =
			\begin{pmatrix}
				\cdot & \cdot & \cdot & \cdot & \ast \\
				\cdot & \ast & \cdot & \cdot & \beta
			\end{pmatrix} \to
			\begin{pmatrix}
				\cdot & \cdot & \cdot & \cdot & \ast \\
				\cdot & \ast & \beta & \cdot & \cdot 
			\end{pmatrix}
		\]}
		\begin{align*}
			\gamma &=
			\begin{pmatrix}
				\cdot & \cdot & \cdot & \cdot & \cdot & \ast \\
				\cdot & \ast & \cdot & \alpha & \alpha & \beta
			\end{pmatrix} \to
			\begin{pmatrix}
				\cdot & \cdot & \cdot & \cdot & \cdot & \ast \\
				\cdot & \ast & \beta & \cdot & \alpha & \alpha
			\end{pmatrix} \\
			&\quad \to
			\begin{pmatrix}
				\cdot & \cdot & \cdot & \cdot & \cdot & \ast \\
				\cdot & \ast & \alpha & \beta & \cdot & \alpha 
			\end{pmatrix} \to
			\begin{pmatrix}
				\cdot & \cdot & \cdot & \cdot & \cdot & \ast \\
				\cdot & \ast & \alpha & \alpha & \beta & \cdot
			\end{pmatrix}
		\end{align*}
		In both of these cases, the arrows exist since $\beta \neq \ast$ and $\alpha \neq \ast$.
		
		Otherwise, both occurrences of $\ast$ are in the top row. If the occurrences of $\alpha$ are not consecutive, then
		{\small\[
			\eta =
			\begin{pmatrix}
				\cdot & \cdot & \ast & \cdot & \ast \\
				\cdot & \cdot & \cdot & \cdot & \beta
			\end{pmatrix} \to
			\begin{pmatrix}
				\cdot & \cdot & \beta & \ast & \cdot & \ast \\
				& \cdot & \cdot & \cdot & \cdot 
			\end{pmatrix}
		\]}
		{\small\[
			\gamma =
			\begin{pmatrix}
				\cdot & \cdot & \ast & \cdot & \cdot & \ast \\
				\cdot & \cdot & \cdot & \cdot & \alpha & \beta
			\end{pmatrix} \to
			\begin{pmatrix}
				\cdot & \cdot & \beta & \ast & \cdot & \cdot & \ast \\
				& \cdot & \cdot & \cdot & \cdot & \alpha
			\end{pmatrix} \to
			\begin{pmatrix}
				\cdot & \cdot & \beta & \alpha & \ast & \cdot & \cdot & \ast \\
				& & \cdot & \cdot & \cdot & \cdot
			\end{pmatrix}
		\]}
		Contrarily, if they are consecutive, then
		{\small\[
			\eta =
			\begin{pmatrix}
				\cdot & \cdot & \ast & \cdot & \ast \\
				\cdot & \cdot & \cdot & \cdot & \beta
			\end{pmatrix} \to
			\begin{pmatrix}
				\cdot & \cdot & \beta & \ast & \cdot & \ast \\
				& \cdot & \cdot & \cdot & \cdot 
			\end{pmatrix}
		\]}
		{\small\begin{align*}
			\gamma &=
			\begin{pmatrix}
				\cdot & \cdot & \ast & \cdot & \cdot & \ast \\
				\cdot & \cdot & \cdot & \alpha & \alpha & \beta
			\end{pmatrix} \to
			\begin{pmatrix}
				\cdot & \cdot & \beta & \ast & \cdot & \cdot & \ast \\
				& \cdot & \cdot & \cdot & \alpha & \alpha
			\end{pmatrix} \\
			&\quad \to
			\begin{pmatrix}
				\cdot & \cdot & \beta & \alpha & \ast & \cdot & \cdot & \ast \\
				& & \cdot & \cdot & \cdot & \alpha
			\end{pmatrix} \to
			\begin{pmatrix}
				\cdot & \cdot & \beta & \alpha & \alpha & \ast & \cdot & \cdot & \ast \\
				& & & \cdot & \cdot & \cdot
			\end{pmatrix}
		\end{align*}}
		and, in both cases, the arrows exist since, as a top arrow starts at $\tau$, there exists a duplicate letter in the bottom row of $\tau$ which is different from $\beta$ (and $\alpha$).
		
		Observe that, in all cases, $\alpha$ cannot be at the end of a row of $\pi'$ and at least one of its occurrences cannot be at the beginning of a row. Thus, $\pi'$ is a simple extension of $\tau'$.
		
		Case \eqref{case2} is completely analogous to case \eqref{case1}, so we will now prove case \eqref{case3}. If $\eta$ is of top type and the winning letter $\ast$ occurs on both rows,
		{\small\[
			\eta =
			\begin{pmatrix}
				\cdot & \cdot & \cdot & \cdot & \ast \\
				\cdot & \ast & \cdot & \cdot & \beta
			\end{pmatrix} \to
			\begin{pmatrix}
				\cdot & \cdot & \cdot & \cdot & \ast \\
				\cdot & \ast & \beta & \cdot & \cdot 
			\end{pmatrix}
		\]}
		{\small\[
			\gamma =
			\begin{pmatrix}
				\cdot & \cdot & \alpha & \cdot & \cdot & \ast \\
				\cdot & \ast & \cdot & \alpha & \cdot & \beta
			\end{pmatrix} \to
			\begin{pmatrix}
				\cdot & \cdot & \alpha & \cdot & \cdot & \ast \\
				\cdot & \ast & \beta & \cdot & \alpha & \cdot 
			\end{pmatrix}
		\]}
		and the arrow exists because $\beta \neq \ast$.
		
		Oh the other hand, if $\ast$ occurs twice in the top row
		{\small\[
			\eta =
			\begin{pmatrix}
				\cdot & \cdot & \ast & \cdot & \ast \\
				\cdot & \cdot & \cdot & \cdot & \beta
			\end{pmatrix} \to
			\begin{pmatrix}
				\cdot & \cdot & \beta & \ast & \cdot & \ast \\
				& \cdot & \cdot & \cdot & \cdot 
			\end{pmatrix}
		\]}
		{\small\[
			\gamma =
			\begin{pmatrix}
				\cdot & \alpha & \cdot & \ast & \cdot & \ast \\
				\cdot & \cdot & \alpha & \cdot & \cdot & \beta
			\end{pmatrix} \to
			\begin{pmatrix}
				\cdot & \alpha & \cdot & \beta & \ast & \cdot & \ast \\
				& \cdot & \cdot & \alpha & \cdot & \cdot 
			\end{pmatrix}
		\]}
		and the arrow exists since, as a top arrow starts at $\tau$, there exists a duplicate letter in the bottom row of $\tau$ which is different from $\beta$.
		
		Observe that, in both cases, $\alpha$ cannot be at the end of a row of $\pi'$, since $\alpha$ is not the next-to-last letter in the bottom row. Thus, $\pi'$ is a simple extension of $\tau'$.
		
		If $\eta$ is an arrow of bottom type, the computations are analogous.
	\end{proof}
	
	\begin{rem}
		When a sequence of simple extensions is clear from context, we will also call $\mathcal{E}_*$ the \emph{composition} of the extension maps of these simple extensions.
	\end{rem}
	
	\section{Rauzy--Veech groups} \label{sec:rauzy-veech}
	
	\subsection{The ``plus'' Rauzy--Veech group}
	
	For a letter $\alpha \in \A$, we define the integers $i_\alpha, j_\alpha \in \{1, \dotsc, 2d\}$ to satisfy $\pi^{-1}(\alpha) = \{i_\alpha, j_\alpha\}$, with $i_\alpha < j_\alpha$. We define an alternate form $\Omega_\pi$ indexed by $\A \times \A$ as:
	{\small\[
		(\Omega_\pi)_{\alpha\beta} =
		\begin{cases}
			+1 & i_\alpha < i_\beta \leq \ell \text{ and } j_\alpha > j_\beta > \ell \\
			+1 & i_\alpha < i_\beta < j_\alpha < j_\beta \leq \ell \\
			+1 & i_\beta < i_\alpha < j_\beta \leq \ell < j_\alpha \\
			+1 & j_\alpha > j_\beta > i_\alpha > \ell \text{ and } i_\alpha > i_\beta \\
			-1 & i_\beta < i_\alpha \leq \ell \text{ and } j_\beta > j_\alpha > \ell \\
			-1 & i_\beta < i_\alpha < j_\beta < j_\alpha \leq \ell \\
			-1 & i_\alpha < i_\beta < j_\alpha \leq \ell < j_\beta \\
			-1 & j_\beta > j_\alpha > i_\beta > \ell \text{ and } i_\beta > i_\alpha \\
			0 & \text{otherwise.}
		\end{cases}
	\]}
	This alternate form is the intersection form of simple closed curves $\{\theta_\alpha\}_{\alpha \in \A}$ joining the midpoints of the sides $M_\pi$, oriented left to right, right to left and upwards for sides that are both in the top row, both in the bottom row, and in both rows, respectively. These curves are a natural basis of $H_1(M_\pi \setminus \Sigma_\pi)$. In particular, if the genus of $M_\pi$ is $g$, then the rank of $\Omega_{\pi}$ is $2g$. We denote $\langle \cdot, \cdot \rangle$ symplectic form on $H_1(M_\pi \setminus \Sigma_\pi)$ defined by $\Omega_\pi$.
	
	 \begin{figure}
	 	\includegraphics[width=\textwidth]{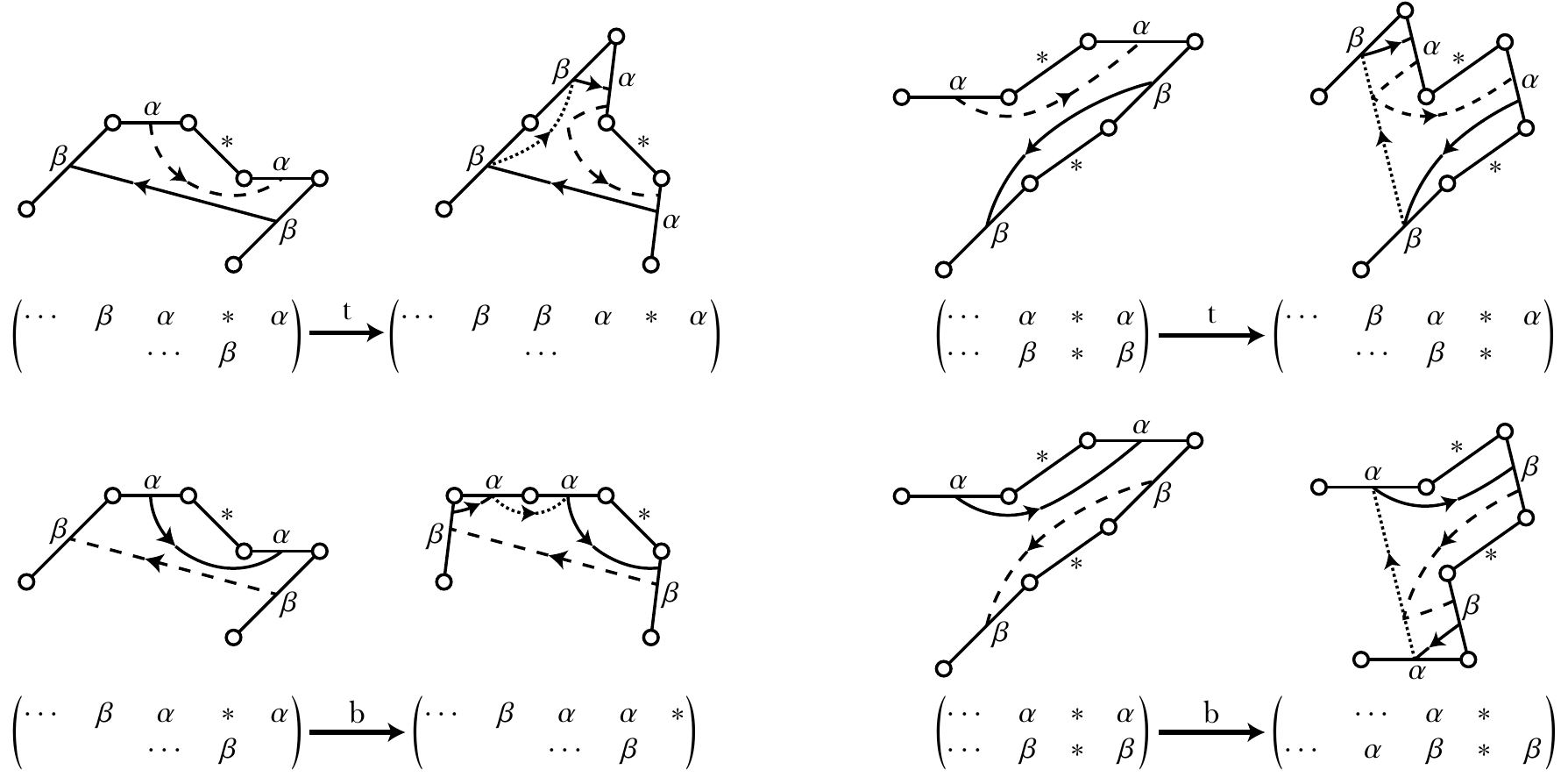}
	 	\caption{Computation of the Kontsevich--Zorich matrices as a change of basis in homology when $\langle \theta_{\alpha_{\mathrm{l}}}, \theta_{\alpha_{\mathrm{w}}}\rangle = 0$. The solid line represents $\theta_{\alpha_{\mathrm{l}}}$, the dashed line $\theta_{\alpha_{\mathrm{w}}} = \theta_{\alpha_{\mathrm{w}}}'$ and the dotted line $\theta_{\alpha_{\mathrm{l}}}'$. Observe that, in all of these cases, $\theta_{\alpha_{\mathrm{l}}} + \theta_{\alpha_{\mathrm{l}}}' + \theta_{\alpha_{\mathrm{w}}} = 0$, so the map $H_1(M_{\pi'} \setminus \Sigmaaux_{\pi'}) \to H_1(M_\pi \setminus \Sigmaaux_\pi)$ is represented in these bases as $\Id - E_{\alpha_{\mathrm{l}}\alpha_{\mathrm{w}}} - 2E_{\alpha_{\mathrm{l}}\alpha_{\mathrm{l}}}$.}
	 	\label{fig:homology}
	 \end{figure}
	
	Let $\RR$ be a Rauzy class of generalized permutations. We consider an undirected version $\tilde{\RR}$ of $\RR$: for each arrow $\gamma = \pi \to \pi'$ we add a reversed arrow $\gamma^{-1} = \pi' \to \pi$. Now, let $\gamma = \pi \to \pi'$ be an arrow in $\RR$. Let $\alpha_{\mathrm{w}}$ and $\alpha_{\mathrm{l}}$ be, respectively, the winner and loser of the operation sending $\pi$ to $\pi'$. We define the \emph{``plus'' Kontsevich--Zorich matrix} indexed by $\A \times \A$ as the change of basis matrix mapping $\{\theta_\alpha'\}_{\alpha \in \A}$ to $\{\theta_\alpha\}_{\alpha \in \A}$, where $\{\theta_\alpha'\}_{\alpha \in \A}$ is the natural basis of $H_1(M_{\pi'} \setminus \Sigma_{\pi'})$ (defined in an analogous way as $\{\theta_\alpha\}_{\alpha \in \A}$ for $M_\pi$). The resulting matrices are:
	\[
		B_{\gamma} = 
		\begin{cases}
			\Id + E_{\alpha_{\mathrm{l}}\alpha_{\mathrm{w}}} & \langle \theta_{\alpha_{\mathrm{l}}}, \theta_{\alpha_{\mathrm{w}}}\rangle \neq 0 \\
			\Id - E_{\alpha_{\mathrm{l}}\alpha_{\mathrm{w}}} - 2E_{\alpha_{\mathrm{l}}\alpha_{\mathrm{l}}} & \langle \theta_{\alpha_{\mathrm{l}}}, \theta_{\alpha_{\mathrm{w}}}\rangle = 0
		\end{cases}
	\]
	 where $\Id$ is the identity matrix and $E_{\alpha\beta}$ has only one non-zero coefficient, equal to $1$, at position $\alpha\beta$. See \Cref{fig:homology} for some of the computations. Observe that the determinant of such matrices is $\pm 1$.
	 
	 \begin{rem}
	 	The action of the Rauzy--Veech induction on suspension data, which was analysed by Boissy and Lanneau \cite{BL:rauzy_veech}, is represented by elementary matrices of the form $\Id + E_{\alpha_{\mathrm{l}}\alpha_{\mathrm{w}}}$. The ``discrepancy'' with the formula for $B_\gamma$ above is explained by the fact that suspension data is always oriented ``rightwards'', that is, $\Re(\zeta_\alpha) > 0$ for each suspension datum $(\zeta_\alpha)_{\alpha \in \A}$. On the other hand, the homological action has to take the changes in orientation into account.
	 \end{rem}
	 
	 Analogously, we define $B_{\gamma^{-1}} = B_{\gamma}^{-1}$. Now consider a walk $\gamma = \gamma_1 \gamma_2 \dotsb \gamma_n$ in $\tilde{\RR}$ starting at $\pi$ and ending at $\pi'$. We define $B_\gamma = B_{\gamma_n} B_{\gamma_{n-1}} \dotsb B_{\gamma_1}$, which satisfies $\Omega_{\pi'} = B_\gamma \Omega_\pi B_\gamma^\tr$. In particular, if $\pi' = \pi$ (that is, if $\gamma$ is a cycle), one has that $B_\gamma$ (acting on \emph{row} vectors) belongs to $\Sp(\Omega_\pi, \Z)$. The ``plus'' Rauzy--Veech group of $\pi$ is the group generated by matrices of this form:
	\begin{defn}
		Let $\RR$ be a Rauzy class and $\pi \in \RR$ be a fixed vertex. We define the \emph{``plus'' Rauzy--Veech group} $\RV^+(\pi)$ as the set of matrices of the form $B_\gamma \in \Sp(\Omega_\pi, \Z)$ where $\gamma$ is a cycle on $\tilde{\RR}$ with endpoints at $\pi$. We will always consider the action of $\RV^+(\pi)$ on row vectors.
	\end{defn}
	
	Observe if $\pi, \pi'$ are vertices of the same Rauzy class $\RR$, then $\RV^+(\pi)$ and $\RV^+(\pi')$ are isomorphic, so we can define the ``plus'' Rauzy--Veech group of a Rauzy class. Indeed, if $\gamma$ is any walk joining $\pi$ and $\pi'$, then the conjugation by $B_\gamma$ is an isomorphism between $\Sp(\Omega_\pi, \Z)$ and $\Sp(\Omega_{\pi'}, \Z)$ and between $\RV^+(\pi)$ and $\RV^+(\pi')$. This shows, in particular, that the ``plus'' Rauzy--Veech group of a Rauzy class has a well-defined index inside its ambient symplectic group.
	
	In general, ``plus'' Rauzy--Veech groups are symplectic with respect to a degenerate symplectic form $\Omega_\pi$. Therefore, we will define Rauzy--Veech groups in ``absolute homology'' in analogy with the Abelian case \cite[Section 6]{G:zariski}. That is, there exist natural maps
	\[
		H_1(M_\pi \setminus \Sigma_\pi) \to H_1(M_\pi) \to H_1(M_\pi, \Sigma_\pi).
	\]
	The former is surjective and is induced by ``forgetting the punctures'', while the latter is injective and comes from the fact that an absolute cycle is also a relative cycle. The natural basis $\{e_\alpha\}_{\alpha \in \A} = \{[\theta_\alpha]\}_{\alpha \in \A}$ of $H_1(M_\pi \setminus \Sigma_\pi)$ is dual to a basis $\{f_\alpha\}_{\alpha \in \A}$ on $H_1(M_\pi, \Sigma_\pi)$ defined by the sides of the polygon $P_\pi$, oriented so that the rightmost copy is goes from left to right. The (possibly degenerate) symplectic form $\Omega_\pi$ descends to the quotient $H_1(M_\pi \setminus \Sigma_\pi) / \ker \Omega_\pi$ into a non-degenerate symplectic form. This group is naturally isomorphic to $H^1(M_\pi)$. The image $V^+(\pi)$ of $H_1(M_\pi)$ inside $H_1(M_\pi, \Sigma_\pi)$ is naturally isomorphic to $H_1(M_\pi)$ and is spanned by $\{\Omega_\pi f_\alpha\}_{\alpha \in \A}$ (where $\Omega_\pi$ is seen as matrix in the basis $\{e_\alpha\}_{\alpha \in \A}$ and $f_\alpha$ is seen as a canonical vector).

	The symplectic isomorphism between $V^+(\pi)$ and $H_1(M_\pi \setminus \Sigma_\pi) / \ker \Omega_\pi$ arising from Poincaré-duality can be described as $\Omega_\pi f_\alpha \mapsto e_\alpha$ for each $\alpha \in \A$. Moreover, the group $\RV^+(\pi)$ induces a \emph{right} action on $H_1(M_\pi \setminus \Sigma_\pi) / \ker \Omega_\pi$ and a \emph{left} action on $V^+(\pi)$ which are dual to each other. Therefore, we define the ``plus'' Rauzy--Veech group in ``absolute homology'' by its action on $H_1(M_\pi \setminus \Sigma_\pi) / \ker \Omega_\pi$. We denote this group by $\RV^+(\pi)|_H \leq \Sp(\Omega_\pi, \Z)|_H$, where the latter is the group of symplectic automorphisms on $H_1(M_\pi \setminus \Sigma_\pi) / \ker \Omega_\pi$. The previous discussion shows that the action on $V^+(\pi)$ produces an isomorphic group.
	
	Moreover, as it is also the case for strata of Abelian differentials, simple extensions allow us to find copies of simpler ``plus'' Rauzy--Veech groups inside more complex ones:
	\begin{lem} \label{lem:extension_matrices}
		Let $\pi$ be a genus-preserving simple extension of $\tau$ obtained by inserting a letter $\alpha$. Then, the canonical injection $\iota \colon H_1(M_{\tau} \setminus \Sigma_{\tau}) \to H_1(M_\pi \setminus \Sigma_{\pi})$ descends to the quotients by $\ker \Omega_\tau$ and $\ker \Omega_\pi$, respectively, into a symplectic isomorphism. Moreover, it conjugates the actions of $\Sp(\Omega_{\tau}, \Z)|_H$ and $\Sp(\Omega_{\pi}, \Z)|_H$, and also the actions of $\RV^+(\tau)|_H$ and a subgroup of $\RV^+(\pi)|_H$.
	\end{lem}
	\begin{proof}
		First observe that $H_1(M_\tau \setminus \Sigma_\tau) / \ker \Omega_\tau$ and $H_1(M_\pi \setminus \Sigma_\pi) / \ker \Omega_\pi$ have the same rank as the genus is preserved. Therefore, $\iota$ maps $\ker \Omega_\tau$ into $\ker \Omega_\pi$ and descends to a well-defined symplectic isomorphism which we will also denote by $\iota$.

		Now, let $\eta$ be an arrow in the Rauzy class of $\tau$ starting at $\tau$ and let $\gamma = \mathcal{E}_*(\eta)$. Assume that $\gamma$ is constructed as case \eqref{case1} in the definition. Let $\beta$ and $\beta'$ be the last letters in the top and bottom rows of $\tau$, respectively. We have several possible cases for the values of $B_\gamma^{-1}$ depending on whether $(\Omega_\pi)_{\beta'\beta} \neq 0$, $(\Omega_\pi)_{\alpha\beta} \neq 0$ and whether the occurrences of $\alpha$ are consecutive in $\pi$, which are listed below:
		{\small\begin{align*}
			(\Id - E_{\beta' \beta})(\Id - E_{\alpha \beta}) &= (\Id - E_{\beta'\beta}) - E_{\alpha \beta} \\
			(\Id - E_{\beta' \beta})(\Id - E_{\alpha \beta} - 2E_{\alpha \alpha}) &= (\Id - E_{\beta'\beta}) + (- E_{\alpha \beta} - 2E_{\alpha\alpha}) \\
			(\Id - E_{\beta' \beta} - 2E_{\beta'\beta'})(\Id - E_{\alpha \beta}) &= (\Id - E_{\beta'\beta} - 2E_{\beta'\beta'}) - E_{\alpha \beta} \\
			(\Id - E_{\beta' \beta} - 2E_{\beta'\beta'})(\Id - E_{\alpha \beta} - 2E_{\alpha \alpha}) &= (\Id - E_{\beta'\beta} - 2E_{\beta'\beta'}) + (- E_{\alpha \beta} - 2E_{\alpha\alpha}), \\
			(\Id - E_{\beta' \beta})(\Id - E_{\alpha \beta})^2 &= (\Id - E_{\beta'\beta}) - 2E_{\alpha \beta} \\
			(\Id - E_{\beta' \beta})(\Id - E_{\alpha \beta} - 2E_{\alpha \alpha})^2 &= \Id - E_{\beta'\beta} \\
			(\Id - E_{\beta' \beta} - 2E_{\beta'\beta'})(\Id - E_{\alpha \beta})^2 &= (\Id - E_{\beta'\beta} - 2E_{\beta'\beta'}) - 2E_{\alpha \beta} \\
			(\Id - E_{\beta' \beta} - 2E_{\beta'\beta'})(\Id - E_{\alpha \beta} - 2E_{\alpha \alpha})^2 &= \Id - E_{\beta'\beta} - 2E_{\beta'\beta'},
		\end{align*}}
		where we used that the three letters $\alpha$, $\beta'$ and $\beta$ are distinct. From these relations, it is easy to see that, in any case, $\iota(u B_{\eta}^{-1}) = \iota(u) B_{\gamma}^{-1}$ for every $u \in H_1(M_\tau \setminus \Sigma_\tau)$. Indeed, one has that $\iota(v) E_{\alpha \beta} = \iota(v) E_{\alpha \alpha} = 0$ for any $v \in H_1(M_\tau \setminus \Sigma_\tau)$ as its $\alpha$-coordinate is $0$. We obtain that:
		\[
			\iota(u) B_\gamma^{-1} = \begin{cases}
				\iota(u)(\Id - E_{\beta' \beta}) & \langle \theta_\beta, \theta_{\beta'} \rangle \neq 0  \\
				\iota(u)(\Id - E_{\beta' \beta} - 2E_{\beta' \beta'}) & \langle \theta_\beta, \theta_{\beta'} \rangle = 0.
			\end{cases}
		\]
		On the other hand, by definition,
		\[
			u B_{\eta}^{-1} = \begin{cases}
				u(\Id - E_{\beta' \beta}) & \langle \theta_\beta, \theta_{\beta'} \rangle \neq 0  \\
				u(\Id - E_{\beta' \beta} - 2E_{\beta' \beta'}) & \langle \theta_\beta, \theta_{\beta'} \rangle = 0,
			\end{cases}
		\]
		so $\iota(u B_{\eta}^{-1}) = \iota(u) B_{\gamma}^{-1}$ as the $\alpha$-coordinate of $\iota(u B_{\eta}^{-1})$ is also $0$.
				
		Similar computations for cases \eqref{case2} and \eqref{case3} show that $\iota_*$ is a monomorphism mapping $\RV^+(\tau)|_H$ to a subgroup of $\RV^+(\pi)|_H$.
	\end{proof}
	
	\subsection{The ``minus'' Rauzy--Veech group}	A similar construction can be used to define the ``minus'' Rauzy--Veech group encoding the action of the Rauzy--Veech induction in the homology of a double cover. We will give a \emph{partial} explicit definition of such group which will be enough to prove our results.
	
	We start by recalling the canonical orientable double cover construction of a quadratic differential $M_\pi$. Let $P_\pi^0 = P_\pi$ and $P_\pi^1$ be a translation $-P_\pi$ which is disjoint from $P_\pi$. Let $\tilde{P}_\pi = P_\pi^0 \cup P_\pi^1$ and $p \colon \tilde{P}_\pi \to P_\pi$ be the natural two-to-one covering between $\tilde{P}_\pi$ and $P_\pi$. Let $\iota' \colon \tilde{P}_\pi \to \tilde{P}_\pi$ be the involution exchanging $P_\pi^0$ and $P_\pi^1$ by translation and central symmetries. We label the sides of $\tilde{P}_\pi$ with the alphabet $\A \times \{0, 1\}$ so that the following conditions hold for any $\alpha \in \A$ and $\varepsilon \in \{0, 1\}$:
	\begin{itemize}
		\item $p$ maps an $(\alpha, \varepsilon)$-side of $\tilde{P}_\pi$ to an $\alpha$-side of $P_\pi$;
		\item $\iota'$ maps an $(\alpha, \varepsilon)$-side of $\tilde{P}_\pi$ to an $(\alpha, 1-\varepsilon)$-side of $\tilde{P}_\pi$;
		\item $P_\pi^0$ contains both $(\alpha, 0)$-sides of $\tilde{P}_\pi$ if and only if $\alpha$ occurs in both rows of $\pi$. 
	\end{itemize}
	These conditions ensure that, when identifying equally-labeled sides, one obtains a valid translation surface $\tilde{M}_\pi$, equipped with a well-defined involution $\iota \colon \tilde{M}_\pi \to \tilde{M}_\pi$ induced by $\iota'$, whose quotient $\tilde{M}_\pi/\iota$ is $M_\pi$. See \Cref{fig:cover} for some examples of this construction. Let $\tilde{\Sigma}_\pi = p^{-1}(\Sigma_\pi)$.
	
	\begin{figure}
		\includegraphics[width=0.7\textwidth]{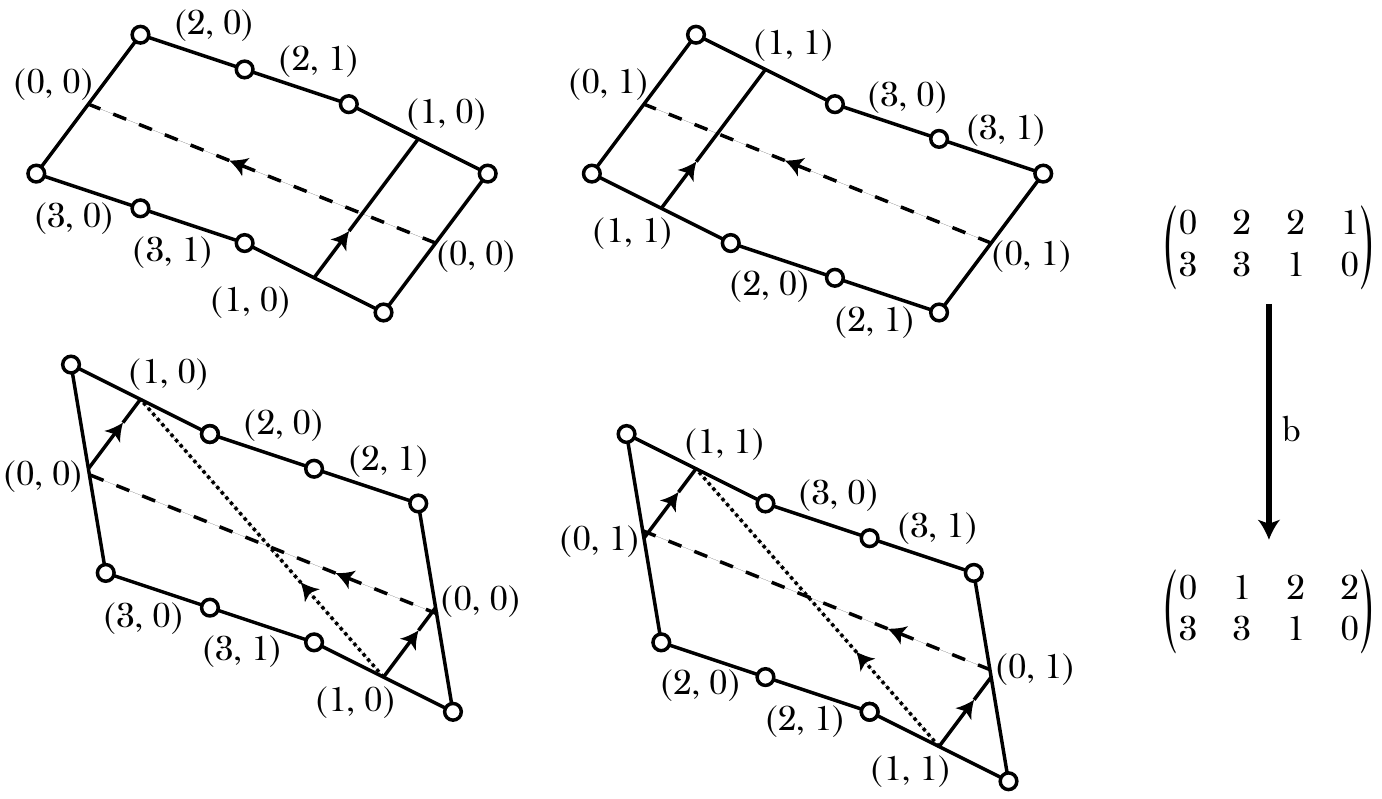}
	 	\caption{Some examples of the double cover construction and a Rauzy operation on it. The solid lines represent $\tilde{\theta}_{\alpha_{\mathrm{l}}}$, the dashed lines $\tilde{\theta}_{\alpha_{\mathrm{w}}} = \tilde{\theta}_{\alpha_{\mathrm{w}}}'$ and the dotted lines $\tilde{\theta}_{\alpha_{\mathrm{l}}}'$. When the winning letter occurs in both rows, the map $S^-(\pi') \to S^-(\pi)$ is represented in these bases as $\Id + E_{\alpha_{\mathrm{l}}\alpha_{\mathrm{w}}}$ if the losing letter belongs to $\mathcal{A}_{\mathrm{tb}}$.}
	 	\label{fig:cover}
	\end{figure}
	
	The involution $\iota$ induces a splitting $H_1(\tilde{M}_\pi \setminus \tilde{\Sigma}_\pi) = H^+(\pi) \oplus H^-(\pi)$ consisting of the elements of $H_1(\tilde{M}_\pi \setminus \tilde{\Sigma}_\pi)$ which are invariant and anti-invariant for $\iota$, respectively. Let $\A_{\mathrm{tb}} \subseteq \A$ be the set of letters occurring in both rows of $\pi$. For $\alpha \in \A_{\mathrm{tb}}$, the cycle $\theta_\alpha \in H_1(M_\pi \setminus \Sigma_\pi)$ lifts to two possible cycles $\tilde{\theta}_\alpha^0, \tilde{\theta}_\alpha^1 \in H_1(\tilde{M}_\pi \setminus \tilde{\Sigma}_\pi)$, which can be though of as belonging to $P_\pi^0$ and $P_\pi^1$, respectively. We define an alternate form indexed by $\A_{\mathrm{tb}}$ as:
	\[
		(\tilde{\Omega}_\pi)_{\alpha\beta} =
		\begin{cases}
			+2 & i_\alpha < i_\beta \text{ and } j_\alpha > j_\beta \\
			-2 & i_\beta < i_\alpha \text{ and } j_\beta > j_\alpha \\
			0 & \text{otherwise,}
		\end{cases}
	\]
	where $i_\alpha < j_\alpha$ and $i_\beta < j_\beta$ are the positions of $\alpha$ and $\beta$ in $\pi$, as above. Clearly, $(\tilde{\Omega}_\pi)_{\alpha\beta} = 2(\Omega_\pi)_{\alpha\beta}$. This matrix is the intersection form of the cycles $\{\tilde{\theta}_\alpha\}_{\alpha \in \A_{\mathrm{tb}}}$, defined as $\tilde{\theta}_\alpha = \tilde{\theta}_\alpha^0 - \tilde{\theta}_\alpha^1$, which belong to $H^-(\pi)$. Let $S^-(\pi)$ be the subspace of $H^-(\pi)$ generated by the $\{\tilde{\theta}_\alpha\}_{\alpha \in \A_{\mathrm{tb}}}$.
	
	Assume that the generalized permutation $\pi$ is obtained by a sequence of simple extensions starting at the permutation $\tau$. By definition of the extension map $\mathcal{E}_*$, a duplicate letter is never the winner of arrows in $\mathcal{E}_*(\eta)$, where $\eta$ is an arrow on the Rauzy class of $\tau$. For this reason, the set of letters occurring in both rows remains the constant while traversing these arrows. We will give a partial definition of the ``minus'' Rauzy--Veech group, in the sense that it does not encode the entire action on $H^-(\pi)$, but rather its action on the subspace $S^-(\pi)$.
	
	As for the ``plus'' case, let $\tilde{\RR}$ be the undirected Rauzy class of $\pi$. Let $\gamma = \pi \to \pi'$ be an arrow in $\RR$ and let $\alpha_{\mathrm{w}}$ and $\alpha_{\mathrm{l}}$ be, respectively, the winner and loser of the operation sending $\pi$ to $\pi'$. We assume that $\alpha_{\mathrm{w}} \notin \A_{\mathrm{tb}}$ as discussed above. We define the \emph{``minus'' Kontsevich--Zorich matrix} indexed by $\A_{\mathrm{tb}} \times \A_{\mathrm{tb}}$ as the change of basis matrix mapping $\{\tilde{\theta}_\alpha'\}_{\alpha \in \A}$ to $\{\tilde{\theta}_\alpha\}_{\alpha \in \A}$, where $\{\tilde{\theta}_\alpha'\}_{\alpha \in \A}$ for $M_{\pi'}$ is defined in an analogous way as $\{\theta_\alpha\}_{\alpha \in \A}$ for $M_\pi$. The resulting matrices are:
	\[
		\tilde{B}_{\gamma} = 
		\begin{cases}
			\Id + E_{\alpha_{\mathrm{l}}\alpha_{\mathrm{w}}} & \alpha_{\mathrm{l}} \in \A_{\mathrm{tb}} \\
			\Id & \alpha_{\mathrm{l}} \notin \A_{\mathrm{tb}}
		\end{cases}
	\]
	See \Cref{fig:cover} for some of the computations. We extend this definition for reversed arrows and walks to define the ``minus'' Rauzy--Veech group:
	\begin{defn}
		Let $\RR$ be a Rauzy class and $\pi \in \RR$ be a fixed vertex. We define the \emph{``minus'' Rauzy--Veech group} $\RV^-(\pi)$ as the set of matrices of the form $\tilde{B}_\gamma \in \Sp(\tilde{\Omega}_\pi, \Z)$ where $\gamma$ is a cycle on $\tilde{\RR}$ with endpoints at $\pi$ such that no duplicate letter is the winner of an arrow of $\gamma$. We will always consider the action of $\RV^-(\pi)$ on row vectors.
	\end{defn}
	
	\begin{rem}
		We have defined the ``minus'' Rauzy--Veech group $\RV^-(\pi)$ restricted to the subspace $S^-(\pi)$ of $H_1(\tilde{M}_\pi \setminus \tilde{\Sigma}_\pi)$ explicitly. We can also define its action on $H_1(\tilde{M}_\pi \setminus \tilde{\Sigma}_\pi)$ in the same way, but we will not use explicit bases or matrices. Nevertheless, in analogy with the ``plus'' case, it is clear that this group preserves an intersection form and that it is well-defined for a Rauzy class. We will see below that, in some cases, our partial definition is sufficient to prove that the index of the ``minus'' Rauzy--Veech group is finite inside its ambient symplectic group.
	\end{rem}
	
	Let $\S$ be a connected component of a stratum of quadratic differentials. Assume that there exists a generalized permutation $\pi$ representing $\S$ obtained by genus-preserving simple extensions from a permutation $\tau$.
	
	If $M_\pi$ belongs to the stratum $\Q(2m_1-1, \dotsc, 2m_s-1, 2m_{s+1}, \dots, 2m_{n})$, then it is well-known that  $\tilde{M}_\pi$ belongs to the stratum $\H(2m_1, \dotsc, 2m_s, m_{s+1}, m_{s+1}, \dotsc, m_n, m_n)$. Thus, if $g$ is the genus of $M_\pi$ and $\tilde{g}$ is the genus of $\tilde{M}_\pi$, we have that
	\[
		2\tilde{g} - 2 = 2\sum_{j = 1}^n m_j = 2\sum_{j = 1}^n m_j - s + s = 4g - 4 + s, \text{ so } \tilde{g} = 2g - 1 + \frac{s}{2}.
	\]
	Let $q \colon H_1(\tilde{M}_\pi \setminus \tilde{\Sigma}_\pi) \to H_1(\tilde{M}_\pi)$ be the map obtained by ``forgetting the punctures''. If $M_\pi$ has exactly two singularities of odd order, then $\tilde{g} = 2g$. Since $\dim H_1(\tilde{M}_\pi) = 2\tilde{g}$, we have that $\dim q(H_1^+(\pi)) = \dim q(H_1^-(\pi)) = 2g$. We conclude that the simple extensions taking $\tau$ to $\pi$ are ``doubly genus-preserving''.
	
	By hypothesis, we have that the rank of $\Omega_\tau$ is $2g$. Moreover, since the alphabet $\B$ of $\tau$ is contained in $\A_{\mathrm{tb}}$ and $(\tilde{\Omega}_\pi)_{\alpha\beta} = 2(\Omega_\pi)_{\alpha\beta}$ for every $\alpha, \beta \in \B$, we have that the rank of $\tilde{\Omega}_\pi$ is also $2g$. Therefore, in analogy with the Abelian and the ``plus'' cases, $\tilde{\Omega}_\pi$ descends to the quotient $S^-(\pi) / \ker \tilde{\Omega}_\pi$ into a non-degenerate symplectic form. We denote the group of symplectic automorphisms of this space by $\Sp(\tilde{\Omega}_\pi, \Z)|_H$. Let $\RV^-(\pi)|_H \leq \Sp(\tilde{\Omega}_\pi, \Z)|_H$ be the group induced by the right action of $\RV^-(\pi)$ on this quotient. We can replicate the proof of \Cref{lem:extension_matrices} to conclude that it is enough to find simple extensions as above to obtain that the index of $\RV^-(\pi)|_H$ is finite in its ambient symplectic group $\Sp(\tilde{\Omega}_\pi, \Z)|_H$.
	\medbreak
	
	\section{Proof of \texorpdfstring{\Cref{thm:main}}{Theorem 1.1}} \label{sec:proof}
	
	By the discussion in the previous section, it is enough to find genus-preserving simple extensions starting at Abelian strata and ending at the desired connected components. The proof is divided in four cases: connected strata, hyperelliptic components, non-hyperelliptic components and exceptional strata. From now on, we will omit the words ``plus'' and ``minus'' when speaking of a Rauzy--Veech group, since the arguments work in both cases (provided that the relevant hypotheses are satisfied).
	
	\subsection{Connected strata} We will start by showing that simple extensions of generalized permutations that are not permutations allow us to ``break up'' the singularities of the meromorphic quadratic differential in any possible way. This lemma is analogous to the case of Abelian differentials \cite[Lemma 6.5]{G:zariski} and will allows us to obtain the proof of \Cref{thm:main} for connected strata.
	
	\begin{lem} \label{lem:extension}
		Let $\tau$ be a generalized permutation $\tau \colon \A \to \{1, \dotsc, 2d\}$. Assume that the surface $M_{\tau} \in \Q(m_1, \dotsc, m_n)$ where $m_1 \geq 1$ and $m_i \geq -1$ for every $2 \leq i \leq n$. Then, there exists an irreducible generalized permutation $\pi$ such that $M_{\pi} \in \Q(m_{1,1}, m_{1,2}, m_2, \dotsc, m_n)$ and such that $\pi$ is a simple extension of $\tau$, where $m_{1,1}, m_{1,2} \geq -1$ are any integers satisfying $m_{1,1} + m_{1,2} = m_1$.
	\end{lem}
	
	\begin{proof}		
		Let $p \colon P_{\tau} \to M_{\tau}$ be the projection map obtained by identifying the sides of the polygon $P_{\tau}$ by translation and rotation.
	
		Consider the bijection $s \colon \{1, \dotsc, 2d\} \to \{1, \dotsc, 2d\}$ defined by:
		\begin{itemize}
			\item $s(k) = \sigma(k-1)$ if $1 < k \leq \ell$;
			\item $s(1) = \sigma(\ell+1)$;
			\item $s(k) = \sigma(k+1)$ if $\ell + 1 \leq k < \ell + m$; and
			\item $s(\ell + m) = \sigma(\ell)$.
		\end{itemize}
		
		The bijection $s$ encodes the process of turning around a marked point in a clockwise manner. Indeed, the set $\{1, \dotsc, \ell\}$ corresponds to the top sides of $P_\tau$, while the set $\{\ell+1, \dotsc, \ell+m\}$ corresponds to the bottom sides. For $1 \leq k \leq \ell$, let $z \in P_\tau$ be its left endpoint. The orbit of $k$ by $s$ is equal to the set of top sides of $P_\tau$ whose left endpoints $z'$ satisfy $p(z') = p(z)$ and the bottom sides of $P_\tau$ whose right endpoints $z'$ satisfy $p(z') = p(z)$.
		
		We can use the orbit by $s$ to compute the conical angle of $p(z)$. Indeed, it is easy to see that such angle is equal to $\tau |\mathrm{Orb}_s(k) \setminus \{ 1, \ell+m \}|$.
		
		Now, let $z \in \C$ be a vertex of $P_{\tau}$ such that $p(z)$ is the conical singularity of order $m_1$. We can assume that $z \neq 0, e_\tau$. Indeed, let $k$ be a top side whose left vertex is $z$ or a bottom side whose right endpoint is $z$. Since $m_1 \geq 2$, one has that $|\mathrm{Orb}_s(k) \setminus \{ 1, \ell+m \}| \geq 3$, so there exists $1 < j \leq \ell$ whose left vertex $z'$ satisfies $p(z') = p(z)$ or $\ell + 1 \leq j < \ell + m$ whose right vertex $z'$ satisfies $p(z') = p(z)$. We replace $z$ by $z'$ if necessary. For the rest of the proof, we will assume that $2 \leq j \leq \ell$, since the other case is analogous.
		
		Consider the set $\mathrm{Orb}_s(j) \setminus \{1, \ell + m\}$, ordered by the order its elements occur when applying  $s$ to $j$ iteratively. Its cardinality is equal to $2 + m_1 = 2 + m_{1,1} + m_{1,2}$. We will consider two cases:
		
		If the $(2 + m_{1,1})$-th element is in the bottom row, then the proof is very similar to the Abelian case. Nevertheless, we will present it in full detail for completeness. Let $\sigma(i) \neq j$ be the $(3 + m_{1,1})$-th element. For a letter $\alpha' \notin \A$, we define the simple extension $\pi$ of $\tau$ by inserting the letter $\alpha'$ just before the letters at positions $i$ and $j$. That is, for every $k \in \{1, \dotsc, 2(d+1)\}$, 
		\[
			\pi(k) = \begin{cases}
				\tau(k) & k < \min\{i, j\} \\
				\alpha' & k = \min\{i, j\} \\
				\tau(k-1) & \min\{i, j\} < k \leq \max\{i, j\} \\
				\alpha' & k = \max\{i,j\} + 1 \\
				\tau(k-2) & \max\{i,j\} + 1 < k \leq \ell + m + 2
			\end{cases}
		\]
		and that the type of $\pi$ is $(\ell', m') = (\ell+1,m+1)$.
		
		Let $\iota \colon \{1, \dotsc, 2d\} \to \{1, \dotsc, 2(d+1)\}$ be the order-preserving injective map such that $\pi \circ \iota = \tau$. Let $i', j'$ be the positions of $\alpha'$ in $\pi$. We choose them so $i' < j'$ if and only if $i < j$.
		
		We will now prove that $M_{\pi} \in \Q(m_{1,1}, m_{1,2}, m_2, \dotsc, m_n)$. Let $\sigma'$ be the involution associated to $\pi$. Consider the bijection $s' \colon \{1, \dotsc, 2(d+1)\} \to\{1, \dotsc, 2(d+1)\}$ defined for $\pi$ in an analogous way as $s$ for $\tau$. We have that:
		
		\begin{itemize}
			\item $s'(j') = \iota(s(j))$;
			\item $s'(j' + 1) = i'$;
			\item $s'(i') = \iota(\sigma(i))$;
			\item $s'(1) = j'$ if $i' = \ell+1$ and $s'(i'-1) = j'$ otherwise;
		\end{itemize}
		and $s'(\iota(k)) = \iota(s(k))$ for any other $k \in \{1, \dotsc, 2d\}$.
		
		Let $k$ be the smallest natural number satisfying $s^{k+1}(j) = \sigma(i)$. If $i = \ell+1$, then $s^k(j) = 1$ and, otherwise, $s^k(j) = i-1$. In any case, $s'(\iota(s^k(j))) = j'$. Moreover, since we have that $s'(j') = \iota(s(j))$ we obtain that the orbit of $j'$ by $s'$ is:
		\[
			j', \iota(s(j)), \iota(s^2(j)), \dotsc, \iota(s^k(j)),
		\]
		so, by the choice of $i$, $|\mathrm{Orb}_{s'}(j') \setminus \{ 1, \ell' + m' \}| = 2 + m_{1,1}$.

		On the other hand, let $k$ be the smallest natural number satisfying $s^{k + 1}(\sigma(i)) = j$. The orbit of $j' + 1$ by $s'$ is:
		\[
			j' + 1, i', \iota(\sigma(i)), \iota(s(\sigma(i))), \iota(s^2(\sigma(i))), \dotsc, \iota(s^k(\sigma(i))),
		\]
		so $|\mathrm{Orb}_{s'}(j'+1) \setminus \{1,\ell'+m'\}| = 2 + m_1 - (2 + m_{1,1}) + 2 = 2 + m_{1,2}$.
		
		These two orbits are disjoint, so the $j'$-side of $M_{\pi}$ joins two distinct conical singularities of orders $m_{1,1}$ and $m_{1,2}$. Since $s'$ coincides with $s$ outside of these orbits, the orders of the rest of the conical singularities are preserved.
		
		Otherwise, if the $(2 + m_{1,1})$-th element is in the top row, let $i$ be such element (which may be equal to $j$). For a letter $\alpha' \notin \A$, we define the simple extension $\pi$ of $\tau$ by inserting the letter $\alpha'$ just before the letters at positions $i$ and $j$ if $i \neq j$, and twice before the letter at position $i = j$ otherwise. The type of this generalized permutation is $(\ell', m') = (\ell+2,m)$
		
		Let $\iota \colon \{1, \dotsc, \ell+m\} \to \{1, \dotsc, \ell + m + 2\}$ be the order-preserving injective map such that $\pi \circ \iota = \tau$. Let $i' \neq j'$ be the positions of $\alpha'$ in $\pi$. We choose them so $i' < j'$ if and only if $i < j$ (in particular, if $i = j$ then $j' < i'$).
		
		We will now prove that $M_{\pi} \in \Q(m_{1,1}, m_{1,2}, m_2, \dotsc, m_n)$. Let $\sigma'$ be the involution associated to $\pi$. Consider the bijection $s' \colon \{1, \dotsc, 2(d+1)\} \to\{1, \dotsc, 2(d+1)\}$ defined for $\pi$ in an analogous way as $s$ for $\tau$. We have that:
		
		\begin{itemize}
			\item $s'(j') = \iota(s(j))$;
			\item $s'(j' + 1) = i'$;
			\item $s'(i') = \iota(s(i))$ if $i \neq j$ and $s'(i') = i'$ otherwise;
			\item $s'(i' + 1) = j'$;
		\end{itemize}
		and $s'(\iota(k)) = \iota(s(k))$ for any other $k \in \{1, \dotsc, \ell + m\}$.
		
		Let $k$ be the smallest natural number satisfying $s^k(j) = i$. Observe that $\iota(i) = i' + 1$, so the orbit of $j'$ by $s'$ is:
		\[
			j', \iota(s(j)), \iota(s^2(j)), \dotsc, \iota(s^k(j)) = i'+1,
		\]
		so, by the choice of $\beta$, $|\mathrm{Orb}_{s'}(j') \setminus \{ 1, \ell' + m' \}| = 2 + m_{1,1}$.
		
		If $i \neq j$, let $k$ be the smallest natural number satisfying $s^k(i) = j$. Since $\iota(j) = j'+1$, the orbit of $i'$ by $s'$ is:
		\[
			i', \iota(s(i)), \iota(s^2(i)), \dotsc, \iota(s^k(i)) = j'+1,
		\]
		so $|\mathrm{Orb}_{s'}(i') \setminus \{1,\ell'+m'\}| = 2 + m_1 - (2 + m_{1,1}) + 2 = 2 + m_{1,2}$.
		
		Otherwise, we have that if $i = j$ and $m_{1,2} = -1$. In this case, the orbit of $i'$ by $s'$ consists only of $i'$. We obtain that $|\mathrm{Orb}_{s'}(i') \setminus \{1,\ell'+m'\}| = 1 = 2 + m_{1,2}$ as desired.
	\end{proof}
	
	Observe that the proof of the previous lemma can also be applied to a (genuine) permutation. Nevertheless, as explained after \Cref{def:extension}, the resulting simple extension will not satisfy \Cref{conv:2.7} unless it is a permutation as well. This is related to the fact that there exist some global obstructions to ``breaking up'' some zeros of even order into two zeros of odd order using local operations \cite{BCGGM:kdifferentials}. This problem can be solved by using the previous lemma two times:
	
	\begin{cor} \label{cor:extension_abelian}
		Let $\tau$ be a permutation $\pi\colon \A \to \{1, \dotsc, 2d\}$. Assume that the surface $M_{\tau} \in \H(m_1, \dotsc, m_n)$ where $m_1 \geq 1$ and $m_i \geq 1$ for every $2 \leq i \leq n$. Then, there exists a generalized permutation $\pi$ such that $M_{\pi} \in \Q(m_{1,1}, m_{1,2}, m_{1,3}, 2m_2, \dotsc, 2m_n)$ and such that $\pi$ is a simple extension of $\tau$, where $m_{1,1}, m_{1,2}, m_{1,3} \geq -1$ are any integers satisfying $m_{1,1} + m_{1,2} + m_{1,3} = 2m_1$ and $m_{1,1}, m_{1,2}$ are odd.
	\end{cor}
	
	\begin{proof}
		We use the previous lemma two times, first adding a duplicate letter in the top row to split the conical singularity of angle $(2 + 2m_1)\pi$ into two conical singularities of angle $(2 + m_{1,1})\pi$ and $(2 + m_{1,2} + m_{1,3})\pi$. This is possible since, as $m_{1,1}$ is odd, the resulting generalized permutation cannot be a permutation. However, this generalized permutation does not satisfy \Cref{conv:2.7}.
		
		Now, we add a duplicate letter in the bottom row to split the conical singularity of angle $(2 + m_{1,2} + m_{1,3})\pi$ into two conical singularities of angle $(2 + m_{1,2})\pi$ and $(2 + m_{1,3})\pi$. Once again, this is possible because $m_{1,2}$ is odd. We obtain a permutation having a duplicate letter in each row, so it satisfies \Cref{conv:2.7}.
	\end{proof}
	
	\begin{rem}
		The strict generalized permutations produced by the previous corollary can be chosen to be irreducible. Indeed, because of the way we count the conical angle in the proof of \Cref{lem:extension}, no letter is inserted at the beginning of a row. Therefore, we can construct a suspension datum $(\zeta_\alpha')_{\alpha \in \A}$ for $\pi$ starting from the ``canonical'' suspension datum $(\zeta_\alpha)_{\alpha \in \B}$ for $\tau$ by setting $\zeta_\alpha' = 1$ if $\alpha \in \A \setminus \B$.
	
		Furthermore, even if the statements of the previous lemma and corollary forbid marked points which are not singularities, it is easy to see that the proof extends to the particular case of the permutation on a two-letter alphabet representing a torus with one marked point. This allows us to treat the genus-1 case below.
	\end{rem}
	
	With these elements, the proof of \Cref{thm:main} is a straightforward consequence of the classification of Rauzy--Veech groups for Abelian strata \cite{AMY:hyperelliptic,G:zariski}:
	
	\begin{proof}[Proof of \Cref{thm:main} for connected strata]
		Let $\S$ be a connected stratum of quadratic differentials of genus $g$ with at least three singularities (zeros or poles), not all of even order. Clearly, there exist at least two singularities of odd order. Therefore, by the previous lemma and the previous corollary, there exists a sequence of simple extensions starting from any connected component of a minimal stratum of Abelian differentials and ending at $\S$. Indeed, we can first use the corollary once to create two singularities of odd order and then use the lemma iteratively to create the remaining singularities.
		
		If $g \geq 4$, then the Rauzy--Veech group of $\S$ contains the Rauzy--Veech groups of both non-hyperelliptic minimal strata of genus-$g$ Abelian differentials. These groups are maximal subgroups of $\Sp(2g, \Z)$ \cite[Theorem 3]{BGP:finiteindex} \footnote{This result uses the classification of finite simple groups.}, so we obtain that the Rauzy--Veech group of $\S$ is $\Sp(2g, \Z)$. If $g = 3$, then we can use the same argument as before, but using the connected components $\H(4)^\odd$ and $\H(4)^\hyp$. The Rauzy--Veech groups of such components have index $28$ and $288$, respectively. Since $28$ does not divide $288$, we conclude that the Rauzy--Veech group of $\S$ is $\Sp(6, \Z)$ by maximality. If $g = 2$, we do not conclude that the Rauzy--Veech group is the entire ambient symplectic group, but we have that it has finite index because it contains a copy of the Rauzy--Veech group of $\H(2)$, which has index $6$ inside $\Sp(4, \Z)$. Finally, if $g = 1$, we can use the adjacency with the stratum $\H(0)$. Indeed, its Rauzy--Veech group is equal to its entire ambient symplectic group, so the same is true for the Rauzy--Veech group of $\S$.
	\end{proof}
	
	\subsection{Hyperelliptic components} As was shown by Lanneau \cite{L:hyperelliptic}, strata of the form $\Q(4j+2, 2k-1, 2k-1)$ and $\Q(2j-1, 2j-1, 2k-1, 2k-1)$ for any integers $j, k \geq 0$ have a hyperelliptic component. The genera of these strata are $j+k+1$ and $j+k$, respectively, and we will assume it to be at least one. We can find an explicit simple extension starting from a hyperelliptic connected component of Abelian strata, therefore showing that the desired Rauzy--Veech groups have finite index. Indeed, for an integer $d \geq 2$ consider the symmetric permutation on $d$ letters:
	{\small\[
		\tau_d = \begin{pmatrix}
			0 & 1 & 2 & \cdots & d-1 \\
			d-1 & d-2 & d-3 & \cdots & 1
		\end{pmatrix}.
	\]}
	Let $g_d = d/2$ if $d$ is even and $(d-1)/2$ if $d$ is odd. It is well-known that this permutation represents the component $\H(2g_d - 2)^\hyp$ if $d$ is even and $\H(g_d - 1, g_d - 1)^\hyp$ if $d$ is odd.
	
	Now, for $s, r \geq 1$ let
	{\small\[
		\sigma_{s, r} =
		\begin{pmatrix}
			0 & A & 1 & 2 & \cdots & s & A & s+1 & s+2 & \cdots & s+r \\
			s+r & \cdots & s+2 & s+1 & B & s & \cdots & 2 & 1 & B & 0
		\end{pmatrix}
	\]}
	which can be obtained from $\tau_{s+r+1}$ by applying two simple extensions. By putting $s = 2k$ and $r = 2j+1$, $\sigma_{s, r}$ represents the connected component $\Q(4j+2,2k-1,2k-1)^\hyp$, and, by putting $s = 2k$ and $r = 2j$, it represents $\Q(2j-1,2j-1,2k-1,2k-1)^\hyp$ \cite[Section 3.6]{Z:representatives}. In the former case, removing the letters $A$ and $B$ produces $\tau_{2j+2k+2}$, whose genus is $j+k+1$. In the latter case, it produces $\tau_{2j+2k+1}$, whose genus is $j+k$. Therefore, we obtain genus-preserving simple extensions starting at hyperelliptic components of Abelian strata and ending at $\Q(4j+2, 2k-1, 2k-1)^\hyp$ or $\Q(2j-1,2j-1,2k-1,2k-1)^\hyp$. We conclude that the indices of their Rauzy--Veech groups are finite in their ambient symplectic groups.
	
	\subsection{Non-hyperelliptic components} All strata of the form $\Q(4j+2, 2k-1, 2k-1)$ and $\Q(2j-1, 2j-1, 2k-1, 2k-1)$, for integers $j, k \geq 0$ and genus at least two, have exactly two connected components, except for $\Q(6,3,3)$ and $\Q(3,3,3,3)$ (which have three), and $\Q(2,1,1)$ and $\Q(1,1,1,1)$ (which are connected). We will consider strata of this form with exactly two connected components and prove that the Rauzy--Veech group of the non-hyperelliptic one has finite index.
	
	We will first prove it for strata of surfaces of genus at least three. Consider the following permutation representatives of minimal Abelian strata computed by Zorich:
	{\small\[
		\tau_g =
		\begin{pmatrix}
			0 & 1 & 2 & 3 & 5 & 6 & \cdots & 3g - 7 & 3g - 6 & 3g - 4 & 3g - 3 \\
			3 & 2 & 6 & 5 & 9 & 8 & \cdots & 3g - 3 & 3g - 4 & 1 & 0
		\end{pmatrix}
	\]}
	for $g \geq 3$ and
	{\small\[
		\sigma_g =
		\begin{pmatrix}
			0 & 1 & 2 & 3 & 5 & 6 & \cdots & 3g - 7 & 3g - 6 & 3g - 4 & 3g - 3 \\
			6 & 5 & 3 & 2 & 9 & 8 & \cdots & 3g - 3 & 3g - 4 & 1 & 0
		\end{pmatrix}
	\]}
	One has that $\tau_g$ represents the strata $\H(2g - 2)^\odd$ for every $g \geq 3$ and that $\sigma_g$ represents $\H(2g - 2)^\even$ for every $g \geq 4$ \cite[Proposition 3, Proposition 4]{Z:representatives}. Moreover, these representatives have a single cylinder. One can use \Cref{lem:extension} and \Cref{cor:extension_abelian} to produce generalized permutation representatives of $\Q(4j+2, 2k-1, 2k-1)$ and $\Q(2j-1, 2j-1, 2k-1, 2k-1)$, which will also have a single cylinder since, as is made explicit in the  proof of \Cref{lem:extension}, we can assume that the letters were not inserted at the beginning of the top row. We will show that such extensions are not hyperelliptic by using the following fact: a generalized permutation such that the first letter in its top row coincides with the last letter in its bottom row represents a hyperelliptic connected component if and only if the generalized permutation obtained by removing said letter has, up to cyclic permutations on both rows, one of the following two forms:
	{\small\[
		\begin{pmatrix}
			A & 1 & 2 & \cdots & s & A & s+1 & s+2 & \cdots & s+r \\
			s+r & \cdots & s+2 & s+1 & B & s & \cdots & 2 & 1 & B
		\end{pmatrix}
	\]}
	or
	{\small\[
		\begin{pmatrix}
			1 & 2 & \cdots & r+1 & 1 & 2 & \cdots & r+1 \\
			r+2 & r+3 & \cdots & r+s+2 & r+2 & r+3 & \cdots & r+s+2
		\end{pmatrix}.
	\]}
	This fact was proven by Lanneau \cite[Proposition 11]{Z:representatives}. It is clear that by inserting two or three letters to $\tau_g$ or $\sigma_g$ it will not have one of those forms, so the extensions lie in the non-hyperelliptic components.
	
	If the genus is three, we get sequences of simple extensions starting from $\H(4)^\odd$, whose Rauzy--Veech group has index $28$ inside its ambient symplectic group, and ending at any non-exceptional and non-hyperelliptic component. Moreover, we can also find simple extensions starting from $\H(4)^\hyp$ and ending at the four components $\Q(10, -1, -1)^\nonhyp$, $\Q(6, 1, 1)^\nonhyp$, $\Q(5, 5, -1, -1)^\nonhyp$, $\Q(3, 3, 1, 1)^\nonhyp$. Indeed, the generalized permutations
	{\small\[
	\begin{pmatrix}
		1 & A & A & 2 & 3 & 4 & 5 & 6 \\
		6 & B & B & 5 & 4 & 3 & 2 & 1
	\end{pmatrix}\quad \textrm{ and } \quad
	\begin{pmatrix}
		1 & A & 2 & 3 & A & 4 & 5 & 6 \\
		6 & B & 5 & 4 & B & 3 & 2 & 1
	\end{pmatrix}
	\]}
	represent the components $\Q(10, -1, -1)^\nonhyp$, $\Q(6, 1, 1)^\nonhyp$, respectively, and erasing the letters $A$ and $B$ produces permutations representing $\H(4)^\hyp$. The other two components can be obtained by using these two generalized permutations together with \Cref{lem:extension}. Since the index of $\H(4)^\hyp$ is $288$, we obtain that the Rauzy--Veech groups of these four components are equal to their entire ambient symplectic groups by maximality.
	
	The only remaining non-exceptional and non-hyperelliptic component in genus three is $\Q(2, 3, 3)^\nonhyp$. In this case, we were not able to find a simple extension starting at $\H(4)^\hyp$, so we only conclude that the Rauzy--Veech group has index at most $28$.
	
	If the genus is greater than three, we get sequences of simple extensions starting from both non-hyperelliptic components of minimal Abelian strata, so the Rauzy--Veech group is equal to its entire ambient symplectic group.
	
	For genus two, only the strata $\Q(6,-1,-1)$ and $\Q(3,3,-1,-1)$ are not connected. The generalized permutations below represent their non-hyperelliptic components. Moreover, erasing the letters $A$ and $B$ produces representatives of $\H(2)$ and $\H(1,1)$, respectively:
	{\small\[
	\begin{pmatrix}
		1 & 2 & 3 & A & A & 4 \\
		4 & 3 & B & B & 2 & 1
	\end{pmatrix}\quad \textrm{ and } \quad
	\begin{pmatrix}
		1 & 2 & A & A & 3 & 4 & 5 \\
		5 & B & B & 4 & 3 & 2 & 1
	\end{pmatrix}.
	\]}
	The index of the Rauzy--Veech group of $\H(2)$ is $6$, while the Rauzy--Veech group of $\H(1,1)$ is equal to its entire ambient symplectic group.
	
	\subsection{Exceptional strata} There exist four exceptional strata that satisfy our hypothesis, namely $\Q(6,3,-1)$ and $\Q(3,3,3,-1)$ in genus $3$ and $\Q(6,3,3)$ and $\Q(3,3,3,3)$ in genus $4$. These strata have two non-hyperelliptic connected components, usually called \emph{regular} and \emph{irregular} \cite{L:connected_components,CM:low_genus}. For these cases, we can find explicit simple extensions, shown in \Cref{table:extensions}, to prove that their Rauzy--Veech groups are equal to their entire ambient symplectic groups. They start from either a connected component of the moduli space of Abelian differentials whose Rauzy--Veech group is its entire symplectic subgroup, or from two different connected components of minimal Abelian strata.
	
	These simple extensions were found by using the \texttt{surface\_dynamics} package for SageMath to obtain explicit representatives for the exceptional connected components. We then performed a depth-first scan on their Rauzy classes.
	
	\begin{longtable}[ht!]{|c|c|c|}
			\hline
			Start & End & Generalized permutation
			\\
			\hline
			$\H(4)^\hyp$ & $\Q(6,3,-1)^\reg$ & {\small $\begin{pmatrix} 1 & 2 & 3 & A & 4 & A & 5 & 6 \\ 6 & 5 & 4 & 3 & 2 & B & B & 1 \end{pmatrix}$}
			\\
			\hline
			$\H(4)^\odd$ & $\Q(6,3,-1)^\reg$ & {\small $\begin{pmatrix} 1 & 2 & 3 & 4 & A & 5 & A & 6 \\ 6 & 4 & B & B & 2 & 5 & 3 & 1 \end{pmatrix}$}

			\\
			\hline
			$\H(4)^\hyp$ & $\Q(6,3,-1)^\irr$ & {\small $\begin{pmatrix} 1 & A & 2 & 3 & 4 & 5 & A & 6 \\ 6 & B & B & 5 & 4 & 3 & 2 & 1 \end{pmatrix}$}
			\\
			\hline
			$\H(4)^\odd$ & $\Q(6,3,-1)^\irr$ & {\small $\begin{pmatrix} 1 & 2 & 3 & 4 & 5 & A & A & 6 \\ 6 & B & 3 & B & 5 & 2 & 4 & 1 \end{pmatrix}$}
			
			\\
			\hline
			$\H(3,1)$ & $\Q(3,3,3,-1)^\reg$ & {\small $\begin{pmatrix} 1 & A & A & 2 & 3 & 4 & 5 & 6 & 7 \\ 7 & 6 & B & 5 & 2 & B & 4 & 3 & 1 \end{pmatrix}$}
			
			\\
			\hline
			$\H(3,1)$ & $\Q(3,3,3,-1)^\irr$ & {\small $\begin{pmatrix} 1 & 2 & 3 & 4 & 5 & A & 6 & A & 7 \\ 7 & 6 & 2 & B & B & 5 & 4 & 3 & 1 \end{pmatrix}$}
			
			\\
			\hline
			$\H(6)^\even$ & $\Q(6,3,3)^\reg$ & {\small $\begin{pmatrix} 1 & 2 & A & 3 & 4 & 5 & 6 & 7 & A & 8 \\ 8 & 7 & 5 & B & 2 & 6 & B & 4 & 3 & 1 \end{pmatrix}$}
			\\
			\hline
			$\H(6)^\odd$ & $\Q(6,3,3)^\reg$ & {\small $\begin{pmatrix} 1 & A & 2 & 3 & 4 & 5 & A & 6 & 7 & 8 \\ 8 & 4 & 7 & B & 5 & 3 & 6 & B & 2 & 1 \end{pmatrix}$}
			
			\\
			\hline
			$\H(6)^\even$ & $\Q(6,3,3)^\irr$ & {\small $\begin{pmatrix} 1 & 2 & 3 & 4 & A & 5 & A & 6 & 7 & 8 \\ 8 & 7 & 5 & B & 2 & 6 & B & 4 & 3 & 1 \end{pmatrix}$}
			\\
			\hline
			$\H(6)^\odd$ & $\Q(6,3,3)^\irr$ & {\small $\begin{pmatrix} 1 & 2 & 3 & 4 & 5 & 6 & A & 7 & A & 8 \\ 8 & B & 5 & B & 3 & 7 & 4 & 6 & 2 & 1 \end{pmatrix}$}
			
			\\
			\hline
			$\H(3,3)^\nonhyp$ & $\Q(3,3,3,3)^\reg$ & {\small $\begin{pmatrix} 1 & A & 2 & A & 3 & 4 & 5 & 6 & 7 & 8 & 9 \\ 9 & 6 & B & 5 & 3 & 7 & 2 & 8 & B & 4 & 1 \end{pmatrix}$}
			
			\\
			\hline
			$\H(3,3)^\nonhyp$ & $\Q(3,3,3,3)^\irr$ & {\small $\begin{pmatrix} 1 & A & 2 & A & 3 & 4 & 5 & 6 & 7 & 8 & 9 \\ 9 & 5 & 2 & 6 & 4 & 3 & B & 8 & B & 7 & 1 \end{pmatrix}$} \\
			\hline
		\caption{Table of extensions for non-hyperelliptic connected components of exceptional strata. The generalized permutation in the third column belongs to the connected component in the second column. Erasing letters $A$ and $B$ produces a permutation belonging to the Abelian connected component in the first column.}
		\label{table:extensions}
	\end{longtable}
	
	\section{Simplicity of the Lyapunov spectra} \label{sec:simplicity}
	
	In this section, we will show that if a fixed stratum $\S$ satisfies the hypothesis of \Cref{thm:main}, then its (``plus'' or ``minus'', depending on the stratum) Lyapunov spectrum is simple. Fix a generalized permutation $\pi$ representing $\S$ and constructed by a sequence of simple extensions starting from a minimal Abelian stratum, as in the proof of \Cref{thm:main}.
	
	We have defined the Rauzy--Veech group of a stratum by allowing the arrows of a Rauzy class to be reversed. The next lemma shows that this coincides with the group generated by \emph{directed} cycles and their inverses.
	\begin{lem}
		The group generated by the (``plus'' or ``minus'') Rauzy--Veech monoid of $\S$ is equal to the (``plus'' or ``minus'', respectively) Rauzy--Veech group of $\S$.
	\end{lem}
	\begin{proof}
		The proof is purely combinatorial and it is therefore identical for the ``plus'' and ``minus'' cases. It is essentially a consequence of the fact that Rauzy classes are strongly connected.
		
		Let $\RR$ be the Rauzy class of $\pi$ and let $\tilde{\RR}$ be an undirected copy. Let $\gamma = \gamma_1^{\varepsilon_1}\gamma_2^{\varepsilon_2} \dotsb \gamma_n^{\varepsilon_n}$ be a cycle in $\tilde{\RR}$, where each $\gamma_i$ is an arrow in $\RR$ and the $\varepsilon_i \in \{-1,+1\}$ are used to denote either an arrow of $\RR$ or its reversed copy. Let $0 = i_0 < i_1 < i_2 < \dotsb < i_k = n$ be indices such that $\varepsilon_{i_j + 1} = \varepsilon_{i_j + 2} = \dotsb = \varepsilon_{i_{j+1}}$ and $\varepsilon_{i_j} \neq \varepsilon_{i_{j+1}}$ for each $0 \leq j < k$. Assume that $\varepsilon_1 = +1$, since the other case is similar.
		
		We define the following oriented cycles:
		\[
			c_1 = \gamma_1 \gamma_2 \dotsb \gamma_{i_1} w_1
		\]
		where $w_1$ is any path in $\RR$ joining the end of $\gamma_{i_1}$ with $\pi$;
		\[
			c_2 = w_2 \gamma_{i_2} \gamma_{i_2 - 1} \dotsb \gamma_{i_1 + 1} w_1
		\]
		where $w_2$ is any path in $\RR$ joining $\pi$ with the start of $\gamma_{i_2}$;
		\[
			c_3 = w_2 \gamma_{i_2 + 1} \gamma_{i_2 + 2} \dotsb \gamma_{i_3} w_3
		\]
		where $w_3$ is any path in $\RR$ joining the end of $\gamma_{i_3}$ with $\pi$; and continue in this way inductively. We choose the last $w_j$ as the zero-length path joining $\pi$ and $\pi$. Then, the matrix $B_\gamma$ is equal to ${}\dotsb B_{c_3} B_{c_2}^{-1} B_{c_1}$ as the matrices $B_{w_j}$ cancel out. We conclude that $B_\gamma$ belongs to the group generated by the Rauzy--Veech monoid of $\pi$, so it contains the Rauzy--Veech group of $\pi$.
	\end{proof}
	
	We obtain that the Rauzy--Veech monoid of $\S$ is Zariski-dense. By the work of Benoist \cite{B:density_pinching}, it is also pinching and twisting.
	
	What remains to obtain the simplicity of the Lyapunov spectrum follows from standard arguments using the general criterion by Avila and Viana \cite{AV:KZ_conjecture}. We will therefore sketch the steps without too many details. First, it is important to consider that not every element of the Rauzy--Veech monoid represents an orbit of the Teichmüller flow. A standard way to ensure that a Teichmüller orbit follows a given Rauzy--Veech orbit is using some ``completeness'' condition. More precisely, we say that a walk $\gamma$ in the Rauzy class of $\pi$ is \emph{$k$-complete} if every letter of $\A$ wins at least $k$ times in $\gamma$. It is clear that $k$-complete walks can be found, as Rauzy classes are strongly connected.  Now, let $\gamma^*$ be a $k$-complete cycle $\gamma^*$ at $\pi$ such that if $\gamma^* = \gamma_{\mathrm{s}}\gamma = \gamma\gamma_{\mathrm{e}}$ then $\gamma$ is either $\gamma^*$ or trivial. By the work of Avila and Resende \cite[Section 6.2]{AR:exponential}, if $k$ is sufficiently large then any cycle $\gamma$ satisfying the following three conditions produces an orbit of the Teichmüller flow:
	\begin{itemize}
		\item $\gamma$ starts with $\gamma^*$;
		\item $\gamma$ ends with $\gamma^*$;
		\item $\gamma$ does not start with $\gamma^*\gamma^*$. 
	\end{itemize}
	We write $\gamma = \gamma^* w \gamma^*$. To conclude, we observe that monoid induced by these cycles is also Zariski-dense, as the entries of $B_{\gamma}$ depend polynomially on the entries of $B_{w}$. This shows that this monoid is pinching and twisting, and, therefore, that the Lyapunov spectrum of the Teichmüller flow on $\S$ is simple, so it concludes the proof.
	
	\medbreak 
	\textbf{Acknowledgements:} I am grateful to Erwan Lanneau for his useful remarks and to Quentin Gendron for observing that some exceptional strata were missing from the proof. I'm also grateful to my advisors, Anton Zorich and Carlos Matheus.
	
\sloppy\printbibliography
	
\end{document}